\documentclass[10pt]{amsart}

\usepackage{amsmath,amsthm,amssymb}
\usepackage{graphicx}
\usepackage[all,cmtip]{xy}
\usepackage{url}
\usepackage{mathrsfs} 
\usepackage[left=1.2in,right=1.2in,top=1.2in,bottom=1.2in]{geometry} 

\newtheorem{theorem}{Theorem}
\newtheorem{corollary}[theorem]{Corollary}
\newtheorem{lemma}[theorem]{Lemma}
\newtheorem{proposition}[theorem]{Proposition}
\newtheorem{definition}[theorem]{Definition}

\theoremstyle{remark}
\newtheorem{remark}[theorem]{Remark}
\newtheorem{claim}[theorem]{Claim}

\numberwithin{equation}{section}
\numberwithin{theorem}{section}

\DeclareMathOperator{\IM}{Im}

\DeclareMathOperator{\GL}{GL}

\DeclareMathOperator{\Gal}{Gal}

\DeclareMathOperator{\Aut}{Aut}
\newcommand{\eps}{\varepsilon}
\newcommand{\To}{\longrightarrow}

\newcommand{\C}{\mathbb{C}}
\newcommand{\Q}{\mathbb{Q}}
\newcommand{\Z}{\mathbb{Z}}

\newcommand{\F}{\mathbb{F}}
\newcommand{\abs}[1]{\left\vert#1\right\vert}
\newcommand{\set}[1]{\left\{#1\right\}}

\newcommand{\symE}{{E^{(2)}}}
\newcommand{\symEone}{{E_1^{(2)}}}
\newcommand{\symbigE}{{\mathcal E^{(2)}}}

\newcommand{\PP}{\mathbb{P}}
\newcommand{\Dsys}{|\mathscr D|}

\DeclareMathOperator{\orth}{O}

\DeclareMathOperator{\Spec}{Spec}

\DeclareMathOperator{\AJ}{AJ}
\DeclareMathOperator{\Alb}{Alb}

\DeclareMathOperator{\Proj}{Proj}

\DeclareMathOperator{\Fr}{Fr}

\title[A surface over $\Q$ with $p_g=q=1$, $K^2=2$ and minimal Picard number]{A surface over $\Q$ with $p_g=q=1$,  $K^2=2$ and \\ minimal Picard number}

\title[Results on surfaces with $p_g=q=1$ and $K^2=2$]{Some results on surfaces with $p_g=q=1$ and $K^2=2$}

\author{Paul Lewis}
\address{Columbia University \\
Department of Mathematics \\
Room 509, MC 4406 \\
2990 Broadway \\
New York, NY 10027}
\email{pdlewis@math.columbia.edu}

\author{Christopher Lyons}
\address{California State University, Fullerton \\
Department of Mathematics \\
800 N. State College Blvd \\
Fullerton, CA 92834 }
\email{clyons@fullerton.edu}


\begin{document}

\maketitle

\begin{abstract}
Following an idea of Ishida, we develop polynomial equations for certain unramified double covers of surfaces with $p_g=q=1$ and $K^2=2$.  Our first main result provides an explicit surface $X$ with these invariants defined over $\Q$ that has Picard number $\rho(X)=2$, which is the smallest possible for these surfaces.  This is done by giving equations for the double cover $\tilde X$ of $X$, calculating the zeta function of the reduction of $\tilde X$ to $\F_3$, and extracting from this the zeta function of the reduction of $X$ to $\F_3$; the basic idea used in this process may be of independent interest.

Our second main result is a big monodromy theorem for a family that contains all surfaces with $p_g=q=1$, $K^2=2$, and $K$ ample.  It follows from this that a certain Hodge correspondence of Kuga and Satake, between such a surface and an abelian variety, is motivated (and hence absolute Hodge).  This allows us to deduce our third main result, which is that the Tate Conjecture in characteristic zero holds for all surfaces with $p_g=q=1$, $K^2=2$, and $K$ ample.
\end{abstract}

\section{Introduction}\label{intro}

Let $X$ be a minimal algebraic surface of general type defined over a finitely generated field $k_0$ of characteristic zero.  Surfaces having geometric genus $p_g=1$ are particularly fascinating, in part because they are related (via the Hodge structure on their middle singular cohomology group) to abelian varieties, as Kuga and Satake \cite{KS} demonstrated.  Among all $p_g=1$ surfaces, the most recognizable are those of Kodaira dimension 0, namely the K3 and abelian surfaces.  The myriad of special features possessed by these two classes of surfaces leads to a wide variety of interesting geometric and arithmetic results about them.  On the other hand, surfaces with $p_g=1$ that are \emph{of general type} lack most of these remarkable features, and much less is known about them.  Indeed, the geometric classification of these surfaces is only partial at this point, and there is a real dearth of arithmetic results in this area.

In this paper we focus on surfaces with $p_g=1$, irregularity $q=1$, and whose canonical bundle has self-intersection number $K^2=2$, with the goal of establishing some results of an arithmetic nature. We give two such results that both, loosely speaking, center around Picard numbers.  The first identifies an explicit surface defined over $\Q$ having $p_g=q=1$, $K^2=2$, and minimum possible Picard number $\rho(X)=2$ (see Theorem \ref{thmA}).  The second establishes the Tate Conjecture in characteristic zero for all surfaces with $p_g=q=1$, $K^2=2$, and $K$ ample (see Theorem \ref{thmC}), which essentially gives a Galois representation-theoretic meaning to their Picard number.

We now describe surfaces $p_g=q=1$ and $K^2=2$ in more detail, beginning with a noteworthy fact about their place in the vast landscape of surfaces of general type.  Among all \emph{irregular} surfaces of general type, one has the two inequalities
\[
K^2\geq 2\chi \geq 2.
\]
The first of these is a consequence of Noether's inequality (see \cite[Prop.\ 2.3.2]{MLP}), while the second holds for all surfaces of general type.  One may inquire about the special collection of those irregular surfaces satisfying the two equalities $K^2 = 2\chi = 2$, and these turn out to be exactly the surfaces with $p_g=q=1$ and $K^2=2$.

These surfaces were classified by Bombieri--Catanese \cite{Cat} and Horikawa \cite{Hor} (see Theorem \ref{class} for a precise statement).  In rough terms, any such surface $X$ is isomorphic to a double cover of the symmetric square $\symE$ of a particular elliptic curve $E$ (namely, $E=\Alb(X)$), with the branching divisor belonging to a specific complete linear system $\Dsys$ on $\symE$.  In a sense, this divisor $\mathscr D$ is easy to describe because $\symE$ has the structure of a $\PP^1$-bundle over $E$; however, the locally free sheaves on $E$ giving rise to this $\PP^1$-bundle are indecomposable, making it difficult to find useful equations for the elements of $\Dsys$.

Instead it turns out to be easier to obtain equations for certain unramified double covers of elements in $\Dsys$.  Indeed, by mimicking a technique of Ishida \cite{Ish} used to study surfaces with $p_g=q=1$ and $K^2=3$, we may pull back the bundle $\symE\to E$ via a $2$-isogeny of elliptic curves $\tilde E\to E$, and obtain a $\PP^1$-bundle $\tilde P\to\tilde E$; the latter bundle will then be the projectivization of a sum of two invertible sheaves of degree one.  As a result, the pullbacks of elements of $\Dsys$ to $\tilde P$ have polynomial equations that are simple enough to work with by hand or in conjunction with computer algebra packages.  Thus if $X$ is a surface with $p_g=q=1$ and $K^2=2$, one may use these equations to acquire explicit information about an unramified double cover $\tilde X$ of $X$; as the results of this paper show, this very often yields useful results about the surface $X$ itself.

For an algebraic surface $S$, let $\rho(S)$ denote its geometric Picard number.  When $X$ has $p_g=q=1$ and $K^2=2$, the canonical divisor and an Albanese fiber are numerically independent, and thus $\rho(X)\geq 2$.  Our first main result is:

\begin{theorem}\label{thmA}
For a certain explicit surface $X_1$ defined over $\Q$ with $p_g=q=1$ and $K^2=2$, we have $\rho(X_1)=2$.
\end{theorem}

\vspace{1em}

We note that the construction of the surface $X_1$ in Theorem \ref{thmA} is ``explicit'' in the sense that we have polynomial equations over $\Q$ for the double cover $\tilde X_1$ and the involution $\iota$, and $X_1$ is then described as the quotient $\tilde X_1/\iota$; see \S\ref{X1_subsec} for the data that determine $X_1$.

In the proof of Theorem \ref{thmA}, the determination of $\rho(X_1)$ is made by passing to a prime $p$ of good reduction.  More precisely, if $\bar X_1$ denotes the good reduction of (an integral model of) $X_1$ to $\F_p$, one has $\rho(X_1)\leq \rho(\bar X_1)$, and one may infer that $\rho(\bar X_1)=2$ if its zeta function $Z(\bar X_1,t)$ has a particular form.  The difficulty in this approach is that one usually obtains $Z(\bar X_1,t)$ by counting points in extensions of $\F_p$, but the lack of equations for $X_1$ makes this infeasible. We overcome this problem by instead working with the double cover $\tilde X_1$.  We note the general idea behind what is described here is simple enough that it can likely be applied in some other situations where one has a finite \'{e}tale Galois cover of the variety of interest.  For ease of notation, let $Y_1:=\tilde X_1$, and let $\bar X_1, \bar Y_1$ denote the reductions of $X_1,Y_1$ to $\F_3$.  By implementing the equations for $Y_1$ in \textsc{Magma} \cite{Magma}, we count points on $\bar Y_1$ to determine $Z(\bar Y_1,t)$. Then geometric relations between $X_1$ and $Y_1$ imply that the zeta functions of their reductions are closely related (see Proposition \ref{zeta_relation}); in fact, based upon the factorization of the rational function $Z(\bar Y_1,t)$ over $\Q$, we are able to narrow down $Z(\bar X_1,t)$ to one of two possibilities.  On the other hand, each of these two possibilities make different predictions for the value of $\#\bar X_1(\F_3)$.  Since $\bar X_1$ is the quotient of $\bar Y_1$ by a free involution $\iota$ (see \cite[Expos\'e 5]{Groth-SGA1} for basic results about such quotients), a simple examination of the $\iota$-orbits of $\bar Y_1(\F_9)$ allows us to definitively determine $\#\bar X_1(\F_3)$ and thus $Z(\bar X_1,t)$.\footnote{We note a contrast here with the problem of finding K3 surfaces $S/\Q$ of minimal Picard number $\rho(S)=1$, as in \cite{vL} or \cite{Els-Jah}.  There one often has explicit equations for $S$, making the determination of $Z(\bar S,t)$ relatively straightforward.  However, due to a basic parity issue (specifically, the second Better number $b_2=22$ is even while the desired Picard number 1 is odd), knowledge of $Z(\bar S,t)$ alone (apparently) never suffices to conclude $\rho(S)=1$.  There is no such parity issue when finding surfaces $X$ with $p_g=q=1$ and $K^2=2$ with minimal Picard number, as here we have $b_2=12\equiv 2 \pmod 2$, and so the main obstacle is the determination of $Z(\bar X,t)$.}

The next main result, our proof of which uses Theorem \ref{thmA} at a key point, regards those surfaces with $p_g=q=1$ and $K^2=2$ for which $K$ is ample.  As described in \cite{Cat}, the classification theorem of Bombieri--Catanese and Horiwaka implies that all such surfaces belong to a smooth projective family $f\colon\mathcal X\to \mathcal S$ over $\Q$, where $\mathcal S$ is a rational variety of dimension 7.  The local system $R^2 f_\ast \Q(1)$ on the complex manifold $\mathcal S(\C)$ gives rise to a monodromy representation $r: \pi_1(\mathcal S(\C),s) \to \Aut\big(H^2(\mathcal X_s,\Q(1))\big)$, for a choice of base point $s$.  By the construction of the family, the image of $r$ fixes the canonical and Albanese classes in $H^2(\mathcal X_s,\Q(1))$.  Letting $\mathbb V_s$ denote the orthogonal complement of these two classes with respect to the cup product form $\psi_s$, it follows that the image of $r$ is contained inside $\orth(\mathbb V_s,\psi_s)$ (included in the obvious way inside $\orth(H^2(\mathcal X_s,\Q(1)),\psi_s)$).  The following ``big monodromy'' result states that there are no other rational algebraic restrictions on the image of $r$:

\begin{theorem}\label{thmB}
In the family $f\colon\mathcal X\to\mathcal S$, the monodromy representation
\[
r\colon \pi_1(\mathcal S(\C),s) \To \orth(\mathbb V_s,\psi_s)
\]
on the middle cohomology of a fiber has Zariski-dense image in $\orth(\mathbb V_s,\psi_s)$.
\end{theorem}

The proof of this theorem uses a strategy similar to that of the second author in \cite{Lyo}, by first making a study of the degenerations of the family $\mathcal X\to S$ within a slightly larger family $\mathcal X_0\to \mathcal S_0$.  Roughly speaking, we show that the singular fibers in $\mathcal X_0\to \mathcal S_0$ are parametrized by a variety $V\subseteq \mathcal S_0$ that has precisely one irreducible component of codimension one in $\mathcal S_0$, and that a general singular fiber in $\mathcal X_0\to \mathcal S_0$ is smooth outside of one ordinary double point.   This description of $\mathcal X_0\to\mathcal S_0$ is deduced from the existence of a certain kind of pencil inside the linear system $\Dsys$ on $\symEone$ (for some elliptic curve $E_1$), and we use \textsc{Singular} \cite{Singular} to find such a pencil by implementing the aforementioned equations  for double covers.  We then apply Picard-Lefschetz theory to a general one-parameter subfamily of $\mathcal X_0\to \mathcal S_0$, and it is at this point that Theorem \ref{thmA} makes its appearance.  Specifically, one would like to know that the subspace $\mathbb V_s$ can be described in terms of the vanishing cycles obtained from this subfamily, and this may be deduced from the fact that a general fiber has Picard number 2, and hence that the associated local system $\mathbb V$ on $\mathcal S(\C)$ (of which $\mathbb V_s$ is a fiber) has no monodromy invariants.

By appealing to a general result in \cite{Lyo}, Theorem \ref{thmB} is then applied to prove our third main result:

\begin{theorem}\label{thmC}
Let $X$ be a surface with $p_g=q=1$ and $K^2=2$ defined over a finitely generated field $k_0$ of characteristic zero, and let $k$ denote the algebraic closure of $k_0$. If the canonical divisor $K$ is ample, then the following hold:
\begin{enumerate}
\item The $\ell$-adic representation of  $\Gal(k/k_0)$ acting upon $H^2(X_k,\Q_\ell(1))$ is semisimple.
\item (Tate Conjecture) Any class of $H^2(X_k,\Q_\ell(1))$ fixed by an open subgroup of $\Gal(k/k_0)$ is a $\Q_\ell$-linear combination of classes arising from divisors on $X_k$.
\end{enumerate}
\end{theorem}

Part (ii) of this theorem may also be phrased as saying that $\rho(X)$ is equal to the dimension of the subspace of classes in $H^2(X_k,\Q_\ell(1))$ that are fixed by an open subgroup of $\Gal(k/k_0)$.  (\emph{A priori}, one only knows that $\rho(X)$ is bounded above by this dimension.)

Let us briefly mention the connection between Theorems \ref{thmB} and \ref{thmC}.  The work of Kuga and Satake \cite{KS}, which was amplified and globalized by Deligne \cite{Del-K3}, gives a certain Hodge correspondence between the middle singular cohomology of a surface with $p_g=1$ and an abelian variety (whose dimension is often quite large).  The framework in \cite[\S5]{Lyo}, which is a generalization of one due to Andr\'e \cite{Andre}, allows one to use the big monodromy result of Theorem \ref{thmB} to show that this Kuga-Satake correspondence is motivated in the sense of Andr\'{e} \cite{Andre-Mot}, and hence absolute Hodge in the sense of Deligne \cite{DMOS}.  Using this in conjunction with the work of Faltings \cite{Fal} on abelian varieties, one is able to deduce Theorem \ref{thmC}.

Since the completing the essential parts of this work, the preprint \cite{Moonen} of Moonen has appeared, which contains a different proof of Theorem \ref{thmC}.  In fact, Moonen's result (see Theorem 9.4(d) of \cite{Moonen}) is more general, in that he does not require $K$ to be ample and he also proves the Mumford-Tate Conjecture for these surfaces.  As here we do here, his method establishes the existence of a motivated correspondence between the surface and an abelian variety 
by placing it within a larger family.  The major difference between our approaches is that Moonen avoids having to explicitly determine the monodromy group of this family.  Hence Theorem \ref{thmB}, which may have further applications to the study of surfaces with $p_g=q=1$ and $K^2=2$, is not proved there.
 
To finish this introduction, let us mention a further plausible relation between the results of Theorems \ref{thmA} and \ref{thmB}. As noted above, we use Theorem \ref{thmA} to identify the global monodromy invariants in a fiber of the local system $R^2 f_\ast \Q(1)$ (see Proposition \ref{dim_I}), which is a key step in our proof of Theorem \ref{thmB}.  In particular, the essential content from Theorem \ref{thmA} used at this step is the following: There exist surfaces $X$ \emph{over $\C$} with $p_g=q=1$ and $K^2=2$ such that $\rho(X)=2$.  But in fact the statement of Theorem \ref{thmB} also implies this existence statement in a nonconstructive manner, due to the relation between the monodromy group and the Mumford-Tate group at a generic member of the family (see, for instance, \cite[Prop.\ 10.14]{PS-book}).  Furthermore, given that the parameter space $\mathcal S$ is rational and has points over $\Q$ (see the proof of Corollary \ref{family}), methods of Terasoma in \cite{Tera} should apply to give the existence (again, nonconstructively) of such surfaces over $\Q$.  From this perspective, one would be able to view Theorem \ref{thmA} as a constructive version of a corollary of Theorem \ref{thmB}.  It seems possible that one could identify the monodromy invariants for $R^2 f_\ast \Q(1)$ without the use of Theorem \ref{thmA}, but we have not checked this.

\vspace{1em}

Here is an outline of the paper.  In \S\ref{background_sec} we summarize the results of Bombieri--Catanese and Horikawa on the classification of surfaces with $p_g=q=1$ and $K^2=2$, and we describe the construction of the family $f\colon\mathcal X\to S$.  The equations for unramified double covers of elements in $\Dsys$ are developed in \S\ref{double_sec}, and two preliminary applications are made: The first shows that the divisor $\mathscr D$ on $\symE$ is not very ample (which necessitates the work done in \S\ref{sing_sec}) and the second gives explicit examples of fibers in $f\colon\mathcal X\to S$ with nongeneric Picard number.  We prove Theorem \ref{thmA} in \S\ref{picard_sec}.  In \S\ref{sing_sec} we use the equations developed in \S\ref{double_sec} to study the collection of singular elements in $\Dsys$.  The results of \S\ref{picard_sec} and \S\ref{sing_sec} allow us to use Picard-Lefschetz theory in \S\ref{PL_mon_sec} to prove Theorem \ref{thmB}.  Finally, we discuss the proof of Theorem \ref{thmC} in \S\ref{tate_sec}.

\subsection{Notation and Terminology} In this paper, $k_0$ will always be a subfield of $\C$ that is finitely generated over $\Q$, and $k$ will denote the algebraic closure of $k_0$ in $\C$.

The term \emph{algebraic surface} will mean a smooth, projective, geometrically irreducible variety of dimension 2 that is minimal.


\section{Preliminaries on Surfaces with $p_g=q=1$ and $K^2=2$}\label{background_sec}

In this section we give background on algebraic surfaces over $k$ with geometric genus $p_g=h^0(\Omega_X^2)=1$, with irregularity $q=h^0(\Omega_X^1)=1$, and whose canonical divisor $K$ has self-intersection number $K^2=2$.  By the Enriques classification of algebraic surfaces (see Table 10 in \cite{BHPV}), these surfaces are of general type.  Via their Albanese maps, they fiber into curves of genus two over an elliptic curve.

These surfaces were classified in characteristic zero by Bombieri--Catanese \cite{Cat} and Horikawa \cite{Hor}, and we will summarize the portions of their work that are most relevant here.  To describe the classification, we set up some notation.  Starting with an elliptic curve $E$ defined over $k$, one may form its symmetric square $\symE$, which is a surface whose points parametrize effective divisors $P+Q$ of degree two on $E$.  The Abel-Jacobi map
\begin{equation}\label{AJ_def}
\AJ\colon \symE\to E, \qquad P+Q \mapsto P\oplus Q
\end{equation}
sends such a divisor to the corresponding sum under the addition law $\oplus$ on $E$.  Letting $O\in E$ denote the group identity, define on $\symE$ the two divisors
\begin{eqnarray*}
D_0 &:=& \set{ O+Q \ | \ Q\in E} \\
G_0 &:=& \AJ^{-1}(O).
\end{eqnarray*}
In terms of the map $\AJ$, which gives $\symE$ the structure of a $\PP^1$-bundle over $E$, the divisor $D_0$ is (the image of) a section and $G_0$ is a fiber.  Using these, we form the divisor
\begin{equation}\label{D_def}
\mathscr D := 6D_0 - 2G_0.
\end{equation}
In \cite{Cat} and \cite{Hor} the following is proved:

\begin{theorem}[Bombieri--Catanese, Horikawa]\label{class}
For any elliptic curve $E$ over $k$, the linear system $\Dsys$ on $\symE$ has (projective) dimension 6 and its general element is smooth.  Furthermore, given any $B\in \Dsys$, let $X_B'$ denote the double cover of $\symE$ branched over $B$ that lives inside the total space of the line bundle associated to $\mathcal O_\symE(3D_0-G_0)$.  If $B$ has at most rational double points, then $X_B'$ is the canonical model of a (minimal) surface $X_B$ with $p_g=q=1$ and $K^2=2$.  In particular, $X_B$ has ample canonical divisor if and only if $B$ is smooth.

Conversely, if $X$ is a surface over $k$ with $p_g=q=1$ and $K^2=2$ with $\Alb(X) \simeq E$, then $X$ is isomorphic to $X_B$ for some $B\in \Dsys$ with at most rational double points.
\end{theorem}

\begin{remark}\label{class_remark} \

\begin{enumerate}
\item Suppose as in the first part of the theorem that we have $B\in|\mathscr D|$ and the associated double cover $c\colon X_B\to\symE$.  If $B$ is smooth, then the (ample) canonical divisor $K$ of $X_B$ is algebraically equivalent to $c^\ast D_0$ and an Albanese fiber of $X_b$ is algebraically equivalent to $c^\ast G_0$.

\item If the elliptic curve $E$ is defined over the subfield $k_0$, then $\symE$ and $\mathscr D$ are as well; hence if $B\in\Dsys$ is defined over $k_0$, so is the surface $X_B$.  In this way, we may produce surfaces with $p_g=q=1$ and $K^2=2$ that are defined over $k_0$. 
\end{enumerate}
\end{remark}

By emulating the constructions of $\symE$, $\AJ$, $D_0$, and $G_0$ above in a relative setting, one may obtain a projective flat family over $\Q$ that contains (among its $k$-fibers) the canonical models of all surfaces over $k$ with $p_g=q=1$ and $K^2=2$.  This idea is outlined in \cite[\S V]{Cat} and we reproduce it now:  

\begin{corollary}\label{family}
There exists a smooth projective family $f\colon\mathcal X\to \mathcal S$ over $\Q$, where $\mathcal S$ is a smooth rational variety of dimension $7$, such that every surface over $k$ with $p_g=q=1,K^2=2$, and $K$ ample is isomorphic to $\mathcal X_s$ for some $s\in \mathcal S(k)$.  Moreover, there exist two effective divisors on $\mathcal X$ which are flat over $\mathcal S$ and whose restrictions to any fiber $\mathcal X_s$ are algebraically equivalent to, respectively, the canonical class and an Albanese fiber.
\end{corollary}

\begin{proof}
Fix an (open) modular curve $A$ of genus 0 that is connected, defined over $\Q$, and serves as a fine moduli space for elliptic curves with some level structure; for instance, one may take $A$ to be modular curve $Y_1(5)$, which also possesses rational points.  Let $\mathcal E\to A$ denote the corresponding universal family and let $O:A\to \mathcal E$ be its identity section.  The product $\mathcal E\times_A\mathcal E$ possesses the $A$-involution which switches the factors, and we obtain the relative symmetric square $\symbigE\to A$ as the quotient by this involution.  Note that since the summation law $\mathcal E\times_A\mathcal E$ is invariant under this involution, we also have the relative Abel-Jacobi map $\mathcal A\mathcal J: \symbigE\to \mathcal E$.

On $\symbigE$, one may define divisors $\mathcal D_0,\mathcal G_0$ which are the relative versions of the divisors $D_0,G_0$ above.  The salient diagram is the following, wherein the triangle and square both commute:
\begin{equation}\label{diagram}
\xymatrix{
A\times \mathcal E \ar[rd] \ar[d]_{O\times \text{Id}} & & \\
\mathcal E\times_A\mathcal E \ar[r]^{\text{quot.}} \ar[d] & \symbigE \ar[d]_{\mathcal A\mathcal J} & \\
A & \mathcal E \ar[l] & A \ar[l]_{O}
}
\end{equation}
Then $\mathcal D_0$ is defined as the image of $A\times\mathcal E\to\symbigE$ at the top of (\ref{diagram}) and $\mathcal G_0$ is defined as the fiber product $\symbigE\times_{\mathcal E} A$ obtained from the bottom right side of (\ref{diagram}).  Moveover, fixing a point $a\in A(k)$ such that $\mathcal E_a\simeq E$, we may base change (\ref{diagram}) via $\Spec k\stackrel{a}{\to} A$ (applied to the copy of $A$ in the lower left of (\ref{diagram})) to conclude that the fibers $\mathcal D_{0,a},\mathcal G_{0,a}$ correspond to the divisors $D_0,G_0$ on $\symE$.

Define $\mathcal L = \mathcal O_{\symbigE}(3\mathcal D_0-\mathcal G_0)$.  For our point $a\in A(k)$, $\mathcal L$ pulls back to $\mathcal O_{\symE}(3D_0-G_0)$ on the fiber $\mathcal E^{(2)}_a\simeq \symE$, and hence $\mathcal L^2$ pulls back to $\mathcal O_{\symE}(\mathscr D)$.  As $h^0(\mathcal O_{\symE}(\mathscr D))=7$ and $h^1(\mathcal O_{\symE}(\mathscr D))=0$ (for all $E$) by \cite[Theorem 1.7]{CC}, we conclude the following (see \cite[\S5]{Mum-AV}): If $\pi\colon \symbigE\to A$ denotes the structure morphism of $\symbigE$ as a $A$-scheme, then $\mathcal F:= \pi_\ast(\mathcal L^2)$ is a locally free sheaf of rank 7 on $A$.  Define the corresponding $\PP^6$-bundle $\mathcal S_0:= \PP(\mathcal F) := \text{Proj}(\text{Sym}(\mathcal F^\vee))$ over $A$.

If $s\in \mathcal S_0$ lies over $a\in A(k)$ and $\mathcal E_a\simeq E$, then $s$ corresponds to an element of $\PP H^0(\symE,\mathcal O_\symE(\mathscr D))$.  In this way, it makes sense to define the following incidence correspondence on $\symbigE\times_A \mathcal S_0$:
\[
\mathcal B_0 := \set{(r,s)\in \symbigE\times_A \mathcal S_0 \ \Big| \ s(r) = 0}.
\]
The fiber of $\mathcal B_0$ over $s$ is a divisor $B$ in the linear system $|\mathscr D|$ on $\symE$.  If $p_1:\symbigE\times_A \mathcal S_0\to \symbigE$ denotes the first projection, one may verify that $\mathcal B_0$ is linearly equivalent on $\symbigE\times_A \mathcal S_0$ to $p_1^\ast(6\mathcal D_0-2\mathcal G_0)$.  Hence if $\mathcal M:=p_1^\ast \mathcal L$, then
\[
\mathcal O_{\symbigE\times_A \mathcal S_0}(\mathcal B_0)\simeq \mathcal M^2.
\]
Inside the line bundle associated to $\mathcal M$, we may form the double cover $\mathcal X_0\to \symbigE\times_A \mathcal S_0$ that is branched over $\mathcal B_0$, which in turn yields the flat family $\mathcal X_0\to \mathcal S_0$.  With $s,E,B$ as above, the fiber $\mathcal X_{0,s}$ is then isomorphic to the double cover $X_B'$ of $\symE$.  Finally, we may define $f\colon\mathcal X\to \mathcal S$ to be the subfamily of smooth fibers in $\mathcal X_0\to\mathcal S_0$.  Note that each step in the construction of $f\colon\mathcal X\to\mathcal S$ may be done with $\Q$-coefficients, so this family is defined over $\Q$.  Moreover, $\mathcal S$ is an open subset of a projective bundle over a curve of 

To finish up, we note by Theorem \ref{class} and the construction of $f\colon\mathcal X\to \mathcal S$ that any surface $X$ over $k$ with $p_g=q=1, K^2=2$, and $K$ ample is isomorphic to $\mathcal X_s$ for some $s\in \mathcal S(k)$.  Moreover, the divisors $\mathcal D_0,\mathcal G_0$ may be pulled back to divisors on $\mathcal X$.  By Remark \ref{class_remark}(1), the first of these restricts to a class in each fiber that is algebraically equivalent to the canonical class ( thus giving a polarization to the family), while the second restricts everywhere to an Albanese fiber.
\end{proof}


\section{Double Covers of Surfaces with $p_g=q=1$ and $K^2=2$}\label{double_sec}

In this section, we obtain explicit equations for certain unramified double covers of surfaces with $p_g=q=1$ and $K^2=2$.  This work is directly inspired by Ishida's investigation in \cite{Ish} of Catanese--Ciliberto surfaces of fiber genus three, which in turn draws on work of Takahashi \cite{Tak}.

We start by fixing an elliptic curve $E$ over $k$.  The Abel-Jacobi map (\ref{AJ_def}) realizes $\symE$ as a $\PP^1$-bundle over $E$.  In fact, $\symE \simeq \Proj(\text{Sym}(\mathcal V'))$ for an indecomposable locally free sheaf $\mathcal V'$ over $E$ of rank 2 and degree 1 (see \cite[p.390]{CC}).\footnote{Note that we are \emph{not} writing $\symE\simeq\PP(\mathcal V'):=\Proj\big(\text{Sym}\big((\mathcal V')^\vee\big)\big)$ here; in that case, $\mathcal V'$ would have rank 2 and degree -1, which is more in line with the notations of \cite{CC}.}  By \cite{Ati}, replacing $\mathcal V'$ by another indecomposable locally free sheaf $\mathcal V$ of rank 2 and degree 1 on $E$ does not change the isomorphism class of $\Proj(\text{Sym}(\mathcal V))$, and we will make a convenient choice of such $\mathcal V$.  Let $\tilde E$ be another elliptic curve over $k$ such that one has an isogeny $\varphi\colon\tilde E\to E$ of degree 2.  Let $\ker\varphi = \{\tilde O,C\}$, where $\tilde O$ is the group identity on $\tilde E$ and $C$ is a 2-torsion point.  Then by \cite{Ati}, the sheaf
\[
\mathcal V:= \varphi_\ast \mathcal O_{\tilde E}(\tilde O)
\]
is an indecomposable locally free sheaf over $E$ of rank 2 and degree 1.

Using this definition of $\mathcal V$, we fix an identification between $\symE$ and $\Proj(\text{Sym}(\mathcal V))$ and define $\tilde P$ via the following fiber product diagram:
\[
\xymatrix{
\tilde P \ar[r]^\Phi \ar[d] & \symE \ar[d]^{\AJ} \\
\tilde E \ar[r]^\varphi & E
}
\]
Being the base change of $\varphi$, the morphism $\Phi$ is the unramified quotient map resulting from a free involution $\iota\colon \tilde P\to \tilde P$.  Explicitly, the involution $\iota$ is the lift to $\tilde P$ of the translation map $Q\mapsto Q+C$ on $\tilde E$.  In particular, $\iota$ interchanges the fibers $\tilde P_{\tilde O}$ and $\tilde P_C$ over $\tilde O$ and $C$.

We note that $\tilde P\simeq \Proj(\text{Sym}(\varphi^\ast \mathcal V))$.  The advantage of introducing $\tilde P$ is that, unlike $\mathcal V$, the sheaf $\varphi^\ast \mathcal V$ is decomposable.  Indeed, one has
\[
\varphi^\ast \mathcal V \simeq \varphi^\ast \varphi_\ast \mathcal O_{\tilde E}(\tilde O) \simeq \mathcal O_{\tilde E}(\tilde O) \oplus \mathcal O_{\tilde E}(C).
\]
Let $Z_0$ and $Z_1$ be those sections in
\[
H^0(\tilde E,\varphi^\ast \mathcal V) \simeq H^0(\tilde E,\mathcal O_{\tilde E}(\tilde O)) \oplus H^0(\tilde E,\mathcal O_{\tilde E}(C))
\]
that, respectively, correspond to the constant function $1$ in the summand $H^0(\tilde E,\mathcal O_{\tilde E}(\tilde O))$ and to the constant function 1 in the summand $H^0(\tilde E,\mathcal O_{\tilde E}(C))$.  Then away from the fibers $\tilde P_{\tilde O}$ and $\tilde P_{C}$, the $\PP^1$-bundle $\tilde P$ is locally trivialized by the relative projective coordinate system $(Z_0:Z_1)$.  In these local coordinates, the involution $\iota$ is expressed as
\begin{equation}\label{involution}
\iota\colon (Q; Z_0:Z_1) \mapsto (Q+C; Z_1:Z_0)
\end{equation}
for any $P\in \tilde E\setminus\{\tilde O,C\}$.

Recalling the divisor $\mathscr D$ on $\symE$ in (\ref{D_def}), the invertible sheaf
\[
\mathcal O_\symE(\mathscr D) \simeq \mathcal O_\symE(6) \otimes \mathcal O_\symE(-2\AJ^{-1}(O))
\]
pulls back to
\[
\Phi^\ast \mathcal O_\symE(\mathscr D) \simeq \mathcal O_{\tilde P}(6) \otimes \mathcal O_\symE(-2\tilde P_{\tilde O}-2\tilde P_C)
\]
on $\tilde P$.  As $\Phi$ is the quotient map of $\tilde P$ under $\iota$, the $\iota$-invariant sections of the latter bundle correspond to the sections of the former (see \cite[\S7, Prop.\ 2]{Mum-AV}):
\begin{equation}\label{pullback_isom}
\Phi^\ast: H^0(\symE,\mathcal O_\symE(\mathscr D)) \  \tilde\To \ \left(H^0(\tilde P,\Phi^\ast \mathcal O_\symE(\mathscr D))\right)^\iota.
\end{equation}
Moreover, given $s\in H^0(\symE,\mathcal O_\symE(\mathscr D))$, the divisor $Z(\Phi^\ast s)\subseteq \tilde P$ is an unramified double cover of the divisor $Z(s)\subseteq \symE$.  Equations for $Z(\Phi^\ast s)$ are simpler are simpler to derive than equations for $Z(s)$, and this will aid the local study of divisors in $\Dsys$.

Choose for $\tilde E$ a Weierstrass equation $y^2=c(x)$, where $c(x)$ is a monic nondegenerate cubic polynomial.  Then one may write in these coordinates $C:=(\alpha,0)$, where $c(\alpha)=0$.   Letting $\beta,\gamma$ be the other two roots of $c(x)$, define the constant $A:=c'(\alpha)=(\alpha-\beta)(\alpha-\gamma)$.  We define on $\tilde E$ the following rational functions:
\begin{eqnarray*}
g_0(Q) &:=& x(Q)-x(C) \\
&=& x-\alpha \\
h_0(Q) &:=& y(Q) \\
&=& y \\
g_1(Q) &:=& x(Q+C)-x(C) \\
&=& \frac{A}{x-\alpha} \\
h_1(Q) &:=& y(Q+C) \\
&=&  -\frac{Ay}{(x-\alpha)^2}.
\end{eqnarray*}
Then the following proposition is readily verified:

\begin{proposition}\label{psi_prop}
The following $7$ sections form a basis for the subspace of $\iota$-invariant sections of the invertible sheaf $\Phi^\ast \mathcal O_\symE(\mathscr D) \simeq \mathcal O_{\tilde P}(6) \otimes \mathcal O_{\tilde P}(-2\tilde P_{\tilde O}-2\tilde P_C)$ on $\tilde P$:
\begin{eqnarray*}
\psi_0 &:=& g_0Z_0^6+g_1 Z_1^6 \\ 
\psi_1 &:=& g_0^2Z_0^6+g_1^2 Z_1^6 \\
\psi_2 &:=& g_0Z_0^5Z_1+g_1Z_0Z_1^5 \\
\psi_3 &:=& h_0Z_0^5Z_1+h_1Z_0Z_1^5 \\
\psi_4 &:=& Z_0^4Z_1^2+Z_0^2Z_1^4 \\
\psi_5 &:=& g_0Z_0^4Z_1^2+g_1Z_0^2Z_1^4 \\
\psi_6 &:=& Z_0^3Z_1^3.
\end{eqnarray*}
\end{proposition}

\begin{definition}\label{Psi_defn}
With regard to the isomorphism (\ref{pullback_isom}), let $\Psi_i\in H^0(\symE,\mathcal O_\symE(\mathscr D))$ be defined by $\Psi_i := (\Phi^{\ast})^{-1}(\psi_i)$ for $i=0,\ldots,6$, so that these seven elements form a basis for $H^0(\symE,\mathcal O_\symE(\mathscr D))$.  For $a:=(a_0:\cdots: a_6)\in \PP^6$, let $B(a) \in \Dsys$ denote the divisor
\[
B(a) := Z\left( \sum_{i=0}^6 a_i \Psi_i\right) \subseteq \symE.
\]
Finally, let $X(a)'$ denote the double cover of $\symE$ inside the line bundle associated to $\mathcal O_\symE(3D_0-G_0)$ that is ramified over $B(a)$.  If $X(a)'$, or equivalently $B(a)$, has only rational double point singularities, let $X(a)$ denote its minimal resolution.
\end{definition}

For $a\in\PP^6$, the divisor $B(a)$ in $\symE$ pulls back to the divisor $\tilde B(a) = Z\left( \sum_{i=0}^6 a_i \psi_i\right) \subseteq \tilde P$ via the unramified double cover $\Phi$.  Thus $\tilde B(a)$ is an unramified double cover of $B(a)$.  One may also pull back the surface $X(a)'$ over $\symE$ to obtain a surface $\tilde X(a)'$ over $\tilde P$; note that $\tilde X(a)'$ is the double cover of $\tilde P$ inside the line bundle associated to $\Phi^\ast \mathcal O_\symE(3D_0-G_0)\simeq \mathcal O_{\tilde P}(3)\otimes\mathcal O_{\tilde P}(-\tilde P_{\tilde O}-\tilde P_C)$ that is ramified over $\tilde B(a)$.  If $X(a)'$ has only rational double point singularities, then the minimal resolution $\tilde X(a)$ of $\tilde X(a)'$ is the pullback from $\symE$ to $\tilde P$ of the double cover $X(a)$ with $p_g=q=1$ and $K^2=2$.  We obtain thus the following Cartesian diagram:
\[
\xymatrix{
\tilde X(a) \ar[r] \ar[d] & X(a) \ar[d] \\
\tilde E \ar[r]^\varphi & E
}
\]

At this point we indicate a result that will be used in \S\ref{picard_sec}:

\begin{lemma}\label{alb_lemma}
We have $\Alb(\tilde X(a)) = \tilde E$.
\end{lemma}

\begin{proof}
Since $\tilde X(a)\to \tilde E$ has connected fibers, it suffices to show that $q(\tilde X(a))\leq 1$.  Assume for contradiction that $q(\tilde X(a)) = q >1$.  If the image of $\tilde X(a)$ inside its Albanese map were a curve, this would mean that $\tilde X(a)\to\tilde E$ factors as $\tilde X(a)\to \Gamma \to \tilde E$, where $\Gamma$ is a curve of genus $q>1$.  But as $\Gamma\to \tilde E$ has degree greater than 1, this would imply the fibers of $\tilde X(a)\to \tilde E$ are disconnected.  Hence the conclusion is that the image of $\tilde X(a)$ in $\Alb (\tilde X(a))$ is a surface.  But then, as the minimality of $X(a)$ implies the minimality of the $\tilde X(a)$, Severi's inequality (see \cite{Pardini}) would imply that $4 = K_{\tilde X(a)}^2\geq 4\chi(\mathcal O_{\tilde X(a)}) = 8$, which is false.  This proves the lemma.
\end{proof}

Since $\tilde X(a)'\to X(a)'$ is an unramified double cover, and since $(Z_0:Z_1)$ are relative projective coordinates on $\tilde P$ away from the fibers over $\tilde O,C$, the Albanese fiber $\tilde X(a)'_Q$ of $\tilde X(a)'$ over the point $Q\in\tilde E\setminus\{\tilde O,C\}$ is isomorphic to the hyperelliptic curves
\begin{equation}\label{away_fibers}
y^2 = \sum_{i=0}^6 a_i \psi_i(Q; Z_0:Z_1),
\end{equation}
regarded as a curve in the weighted projective plane $\PP(1:1:3)$.

Next we consider the equations of $\tilde B(a)$ and $\tilde X(a)'$ near the fiber $\tilde P_{\tilde O}$.  Let $t:=x/y$, which is a local parameter at $\tilde O\in\tilde E$.  It follows from the definition of $Z_0$ that relative projective coordinates for $\tilde P$ in a neighborhood of $\tilde P_{\tilde O}$ are given by $(Z_0':Z_1):=(t^{-1}Z_0:Z_1)$.  Moreover, for each $\psi_i$ in Proposition \ref{psi_prop}, let us define $\chi_i$ by $\psi_i = t^2\chi_i$.  Explicitly, $\chi_i$ is obtained by substituting $Z_0=tZ_0'$ into the equations in Proposition \ref{psi_prop}, expanding the rational functions $g_0, g_1, h_0, h_1$ in terms of $t$, and dividing by $t^2$.  Upon doing so, one obtains
\begin{eqnarray*}
\chi_0 &=& AZ_1^6+O(t^2) \\ 
\chi_1 &=& Z_0'^6+O(t^2) \\
\chi_2 &=& t(Z_0'Z_1(Z_0^4+AZ_1^4))+O(t^3) \\
\chi_3 &=& Z_0'Z_1(Z_0'^4-AZ_1^4)+O(t^2) \\
\chi_4 &=& Z_0'^2Z_1^4+O(t^2) \\
\chi_5 &=& Z_0'^4Z_1^2+O(t^2) \\
\chi_6 &=& t(Z_0'^3Z_1^3).
\end{eqnarray*}
These equations will be used to describe the divisors $\tilde B(a)$ near the fiber $\tilde P_{\tilde O}$; explicitly, $\tilde B(a)$ is given by $Z(\sum_{i=0}^6 a_i\chi_i)$ and thus $\tilde X(a)'$ is isomorphic to $y^2=\sum_{i=0}^6 a_i\chi_i$ near $\tilde P_{\tilde O}$.  In particular, setting $t=0$, the fiber $\tilde X(a)'_{\tilde O}$ of $\tilde X(a)'$ over $\tilde O$ is isomorphic to the hyperelliptic curve
\begin{equation}\label{O_fiber}
y^2 = a_1Z_0'^6 +a_3 Z_0'^5 Z_1+a_5Z_0'^4Z_1^2+a_4 Z_0'^2Z_1^4-a_3AZ_0'Z_1^5+a_0AZ_1^6.
\end{equation}

A similar discussion applies near the fiber $\tilde P_{C}$, but arguments using the symmetry $\iota$ will eliminate the need for specific coordinates and equations there.

\begin{remark}
As is our convention, let $k_0$ be a finitely generated field whose algebraic closure is $k$.  If $\tilde E$, $C$, and $a\in \PP^6$ are defined over $k_0$, then $E\simeq \tilde E/C$, $\symE$, $\tilde P$, $\iota$, $B(a)$, $X(a)'$, $\tilde B(a)$, and $\tilde X(a)'$, along with the previously described morphisms between them, are all defined over $k_0$ as well.  In particular, the fibers of $\tilde X(a)'$ away from $\tilde O,C$ are isomorphic over $k_0$ to the hyperelliptic curves given by (\ref{away_fibers}), and the fibers over $\tilde O$ and $C$ are isomorphic over $k_0$ to the hyperelliptic curve (\ref{O_fiber}).  
\end{remark}

The following result gives a first application of the equations in this section:

\begin{proposition}\label{not_very_ample}
The complete linear system $\Dsys$ on $\symE$ is base point free and ample, but is not very ample.
\end{proposition}

\begin{proof}
The fact that $\Dsys$ is base point free is proved in \cite[Theorem 1.18]{CC}, while ampleness is proved in \cite[Prop.\ 1.14]{CC}.  To show that $\Dsys$ is not very ample, it suffices to show that 
the induced morphism $\symE\to \PP H^0(\symE,\mathcal L)^\vee$ does not have an injective differential everywhere.  Equivalently, we may show the induced morphism $\tilde P\to \PP (H^0(\tilde P,\Phi^\ast\mathcal L)^{\iota})^\vee\simeq \PP^6$ does not have injective differential at one point, and for this we consider the point $(Z_0':Z_1)=(0:1)$ in the fiber $\tilde P_{\tilde 0}$.  At this point of $\tilde P$, the pair $(t,Z)$ gives local parameters and $\chi_0$ does not vanish.  Now consider the morphism from $\tilde P$ to $\PP^6$ given by
\[
(\psi_0:\cdots:\psi_6) = (\chi_0:\cdots:\chi_6) = \left( 1: \frac{\chi_1}{\chi_0}: \cdots :\frac{\chi_1}{\chi_0} \right).
\]
One may verify easily that $\frac{\partial }{\partial t}\left(\frac{\chi_j}{\chi_0}\right)\Big\vert_{(t,Z)=(0,0)} = 0$ for each $1\leq j\leq 6$, and thus the $6\times 2$ matrix giving the differential at this point contains a column of zeros.
\end{proof}

As a second application, we consider the problem (in contrast to Theorem \ref{thmB}) of finding surfaces $X$ with $p_g=q=1$ and $K^2=2$ with Picard number $\rho(X)>2$.  If the canonical model of $X$ is not smooth (or equivalently if $K$ is not ample), then $X$ will contain $(-2)$-curves and one will have $\rho(X)>2$.  But it is more difficult to find examples with $\rho(X)>2$ and $K$ ample.  The following proposition gives one construction, and is used in the proof of Theorem \ref{thmC} in \S\ref{tate_sec}.  The idea we use here is analogous to a classical construction of Kummer surfaces having a polarization of degree 2.  There one takes a double cover of $\PP^2$ branched over a configuration of six lines, such that these lines are all tangent to a single conic; then one considers the curves in the double cover that arise from the inverse image of the conic.

\begin{proposition}\label{nongen_picard}
For any elliptic curve $E$ over $k$, there exists a surface $X$ with $p_g=q=1, K^2=2$, $K$ ample, and $\Alb(X)\simeq E$ such that $\rho(X)>2$.
\end{proposition}

\begin{proof}
Consider the divisor $D_0$ inside $\symE$ and let $B\in \Dsys$.  We note that $D_0\simeq E$, so in particular $D_0$ is smooth everywhere, and that $B\cdot D_0=4$.  Let $c\colon X_B\to \symE$ denote the associated double cover, and put $d=c^\ast D_0$ and $f=c^\ast F_0$.  We have $d^2=d.f=2$ and $f^2=0$.

Now suppose we can find such $B$ so that (i) $B$ is smooth and (ii) $B$ is tangent to $D_0$ at precisely two points.  Then by (i) $X_B$ will have ample canonical class $K$ and by (ii) we will have $\pi^\ast D_0 = C_1+C_2$ for two curves with $C_1.C_2=2$, $C_i^2=-1$, and $C_i.d=C_i.f=1$.  The Gram matrix of the three divisors $d,f,C_1$ will be nondegenerate, showing that their classes generate a rank 3 subgroup of the Neron-Severi group of $X$, and so the result will follow.

The rest of the proof will show that such $B$ exist, using the equations developed in this section.  First we have $h^0(\symE,\mathcal O_{\symE}(D_0))=1$ by \cite[Theorem 1.17]{CC}.  Thus the double cover $\tilde D_0$ of $D_0$ inside $\tilde P$ is the vanishing locus of a generator of the space $H^0(\tilde P,\mathcal O_{\tilde P}(1))^{\iota}$, for which we may take the section $Z_0+Z_1$.  Now choose a point $Q=(x_0,y_0)\in\tilde E$ that is not 2-torsion, so that $-Q=(x_0,-y_0)$ is a distinct point.  We consider the linear subsystem $M\subseteq |\Phi^\ast \mathscr D|$ consisting of those elements that are tangent to $\tilde D_0$ at its (unique) points in the fibers $\tilde P_Q$ and $\tilde P_{-Q}$; in the coordinate chart discussed earlier in this section for the open subset $\tilde P\setminus \set{\tilde P_{\tilde 0},\tilde P_C}$, these are the points $(Q; -1:1)$ and $(-Q; -1:1)$.  Assume further that $Q+C\neq -Q$, so that these two points of $\tilde P$ lie over distinct points of $\symE$.  At each of these points of $\tilde P$, we may choose as local parameters $(x-x_0,Z+1)$, where $Z=Z_0/Z_1$.  Using the notation following Definition \ref{Psi_defn}, suppose that $\tilde B(a) \in M$.  Since $\tilde B(a)$ passes through $(Q; -1: 1)$, we have
\[
\sum_{j=0}^6 a_j \psi_j(Q; -1: 1) = 0,
\]
while the tangency condition at this point\footnote{Note that the partial derivative appearing in (\ref{x-deriv}) is to be interpreted as the coefficient $c_{10}$ appearing in the Taylor expansion $\psi_j = c_{00}+c_{10}(x-x_0)+c_{01}(Z+1)+\cdots$ of $\psi_j$ at $(Q; -1 :1)$.
} gives
\begin{equation}\label{x-deriv}
\sum_{j=0}^6 a_j \frac{\partial \psi_j}{\partial x}(Q; -1: 1) = 0.
\end{equation}
Two similar equations result from the tangency of $\tilde B(a)$ to $(-Q; -1: 1)\in \tilde D_0$.  From the equations in Proposition \ref{psi_prop} (and noting in particular that only $\psi_3$ depends on $y$), it follows that $\tilde B(a)\in M$ if and only if $a$ belongs to the kernel of a $4\times 7$ matrix of the form
\[
\begin{bmatrix}
c_0 & c_1 & -c_0 & c_3 & 2 & -c_2 & -1 \\
c_0 & c_1 & -c_0 & -c_3 & 2 & -c_2 & -1 \\
d_0 & d_1 & -d_0 & d_3 & 2 & -d_2 & -1 \\
d_0 & d_1 & -d_0 & -d_3 & 2 & -d_2 & -1 
\end{bmatrix}
\]
for some constants $c_i, d_i$.  As all $4\times 4$ minors of this matrix vanish, its kernel has dimension at least 4.

We will now establish:

\begin{claim}
The only base points of $M$ are $(\pm Q; -1:1)$.
\end{claim}

\begin{proof}
We may easily identify a 3-dimensional projective subspace of $M$: It corresponds to the three independent sections
\begin{eqnarray*}
\psi_0+\psi_2 &=& (Z_0+Z_1)(g_0 Z_0^5+g_1 Z_1^5) \\
\psi_0-\psi_5 &=& (Z_0+Z_1)(Z_0-Z_1)(g_0 Z_0^4-g_1Z_1^4) \\
\psi_4+2\psi_6 &=& (Z_0+Z_1)^2 Z_0^2 Z_1^2,
\end{eqnarray*}
and is the pullback to $\tilde P$ of the 3-dimensional space $|\mathscr D-D_0|$ inside $\Dsys$. Define
\begin{eqnarray*}
s_1 &:=& g_0 Z_0^5+g_1 Z_1^5 \\
s_2 &:=& (Z_0-Z_1)(g_0 Z_0^4-g_1Z_1^4) \\
s_3 &:=& Z_0Z_1,
\end{eqnarray*}
so that
\begin{eqnarray*}
Z(\psi_0+\psi_2) &=& \tilde D_0+Z(s_1) \\
Z(\psi_0-\psi_5) &=& \tilde D_0+Z(s_2) \\
Z(\psi_4+2\psi_6) &=& 2\tilde D_0+2Z(s_3).
\end{eqnarray*}
One may easily verify that, away from the fibers $\tilde P_{\tilde 0}$ and $\tilde P_C$, the only intersection points of $Z(s_1)$ and $Z(s_3)$ are $(\pm Q; -1:1)$.  It follows that if $M$ has any other base points, then some of them must belong to the fiber $\tilde P_{\tilde 0}$.  

Near the fiber $\tilde P_{\tilde 0}$, we recall that $Z_0=t Z_0'$ and thus $\tilde D_0$ has local equation $Z_1+tZ_0'=0$.  Hence the intersection of $\tilde D_0$ with the fiber $\tilde P_{\tilde 0}$ corresponds to the point $(t; Z_0':Z_1) = (0; 1:0)$.  By also using the relation $Z_0=t Z_0'$ and expanding the rational functions $g_0,g_1,h_0,h_1$ in terms of $t$, one may check that the only common intersection point of $Z(s_1),Z(s_2)$ and $Z(s_3)$ inside $\tilde P_{\tilde 0}$ is this same point $(t; Z_0':Z_1) = (0; 1:0)$.  But now let $B\in M$ be any divisor not in the span of $Z(s_1),Z(s_2)$ and $Z(s_3)$; by construction, $B$ only intersects $\tilde D_0$ at two points that are both outside of $\tilde P_{\tilde 0}$.  Hence $M$ has no base points inside $\tilde P_{\tilde 0}$.  This establishes the claim.
\end{proof}

Continuing with the proposition, it follows from Bertini's theorem that a generic element $M$ is smooth away from the two base points $(\pm Q; -1:1)$.  On the other hand, there exist elements of $M$ (such as $Z(\psi_0+\psi_2)$) that are smooth at these base points.  Hence a generic element of $M$ is smooth everywhere.

Taking $\tilde B(a)\in M$ to be generic, we may set $X=X_{B(a)}$ to complete the proof of Proposition \ref{nongen_picard}.
\end{proof}

\begin{remark}\label{period-map}
The existence statement in Proposition \ref{nongen_picard} may be deduced over $\C$ as soon as one knows the variation of Hodge structure $R^2(f_\C)_\ast \Q$ on $\mathcal S_\C$ arising from the family $f:\mathcal X\to\mathcal S$ is nontrivial.  Indeed, in this case a result due independently to Green (see \cite[Prop.\ 5.20]{Voisin2}) and Oguiso \cite{Oguiso} says that the Noether--Lefschetz locus is dense in $\mathcal S_\C$.  This follows (without circularity) from Theorem \ref{thmB}, but would also follow if one were to establish, say, that the differential of the period map was nonzero at some point.  We thank the referee for this observation.
\end{remark}


\section{A Surface with $p_g=q=1$ and $K^2=2$ Having Minimal Picard Number}\label{picard_sec}


\subsection{Zeta functions of reductions modulo $p$.}  First let us make some notation to be used in this section.  Let $p\in\Z$ be a fixed prime and suppose that $Z$ is a smooth projective variety defined over $\Q$ that has good reduction at $p$.  In this section, $\bar Z$ will denote the special fiber of some smooth projective model of $Z$ over the local ring $\Z_{(p)}$.  In particular, $\bar Z$ will be a smooth projective variety over $\F_p$.  Next, if $W$ is a smooth projective variety over some field $F_0$ with algebraic closure $F$, and $\ell$ is a prime number different from $\text{char}(F)$, the $\ell$-adic cohomology group $H^i(W_{F},\Q_\ell(j))$ will be denoted as $H^i_\ell(W)(j)$ (or just $H^i_\ell(W)$ when $j=0$).

Suppose that $\tilde E_1\to E_1$ is a 2-isogeny of elliptic curves that is defined over $\Q$.  Also suppose that we have fixed $a_1 \in \PP^6(\Q)$ such that the curve $B(a_1)\subseteq E_1^{(2)}$ is smooth.  To ease notation, let us put
\begin{eqnarray*}
B_1 &:=& B(a_1) \\
X_1 &:=& X(a_1) \\
\tilde B_1 &:=& \tilde B(a_1) \\
Y_1 &:=& \tilde X(a_1)
\end{eqnarray*}
Then we may consider the reductions of the data $E_1,\tilde E_1,\varphi, \tilde B_1$ to $\F_p$; we suppose further that all of these objects have good reduction to $\F_p$.\footnote{The integral models we use in our calculations for this section (see \cite{LyData}) simply come from reducing the equations in \S\ref{double_sec} modulo $p$.  As one observes there, the equations have integer coefficients, with the possible exception of the quantities $\alpha$ and $A$ (see the definitions right above Proposition \ref{psi_prop}).  Hence by choosing $\tilde E$ to be the curve $y^2=c(x)$, where the cubic $c(x)$ has integer coefficients and an integer root $\alpha$, it becomes possible to work entirely with polynomials over $\Z$ (perhaps after clearing some factors of $(x-\alpha)$ in denominators that arise from the rational functions $g_1$ and $h_1$).  We may then consider these as polynomials over $\Z_{(p)}$ and $\F_p$.}

We now discuss the zeta functions of the surfaces $\bar X_1$ and $\bar Y_1$; see \cite{Katz} for a general reference. Let $\Fr \in \Gal(\bar \F_p/\F_p)$ denote the geometric Frobenius automorphism.  The zeta function of the smooth projective surface $\bar X_1$ is then
\[
Z(\bar X_1; t) = \frac{P_1(t)P_3(t)}{P_0(t)P_2(t)P_4(t)},
\]
where
\[
Z_i(t) = \det\big(1-t\cdot \Fr^\ast \vert {H^i_\ell(\bar X_1)} \big)
\]
for a prime $\ell\neq p$.   One has $P_0(t) = 1-t$, $P_4(t)=1-p^2t$.  Moreover, the Albanese morphism $\bar X_1\to \bar E_1$ induces an isomorphism of $\Gal(\bar \F_p/\F_p)$-modules $H^1_\ell(\bar E_1)\to H^1_\ell(\bar X_1)$; hence, $P_1(t)$ may be calculated from computing the zeta function $Z(\bar E_1,t)$.  Moreover, Poincar\'e duality gives $P_3(t) = P_1(pt)$.

In light of the preceding comments, the bulk of the work in determining $Z(\bar X_1,t)$ centers upon the computation of the factor $P_2(t)$.  First we recall certain properties that follow from the general theory of $\ell$-adic cohomology and the Weil Conjectures.  Since the 2nd Betti number of $X_1$ is 12, one knows that $P_2(t)\in \Z[t]$ is a polynomial of degree 12, which is typically written in the form
\[
P_2(t) = \prod_{i=1}^{12}(1-\alpha_i t)
\]
for inverse roots $\alpha_i$ that all have complex absolute value $\abs{\alpha_i}=p$ (for any complex embedding).  Any divisor $D$ on $\bar X_1$ has a cycle class $[D]\in H^2_\ell (\bar X_1)(1)$; if $D$ is defined over the finite extension $K/\F_p$, then the class $[D]$ is fixed by the subgroup $\Gal(\bar \F_p/K)$.  In particular, if we look at the subspace of $H^2_\ell(\bar X_1)(1)$ generated by the Neron-Severi group of $\bar X_1\otimes_{F_p} \bar \F_p$, then the automorphism $\Fr\in \Gal(\bar \F_p/\F_p)$ preserves this subspace and the eigenvalues of its action upon this subspace are all roots of unity.  (The Tate Conjecture for $\bar X_1$ over $\F_p$ predicts that this characterizes all such classes $[D]$, but this characteristic $p$ statement is still unknown, and in any case will not be needed here.)  Transferring this back to the untwisted group $H^2_\ell(\bar X_i)$, we obtain the following upper bound for the geometric Picard number of $\bar X_1$ by looking at the inverse roots $\alpha_i$ of $P_2(t)$:
\begin{equation}\label{pic_bound}
\rho(\bar X_1) \leq \#\set{\alpha_i \ : \ \frac{\alpha_i}p \text{ is a root of unity} }.
\end{equation}

In $H^2_\ell(\bar X_1)(1)$, the subspace generated by the cycle classes of the canonical divisor and an Albanese fiber is a two-dimensional subspace that is fixed by $\Fr$.  Indeed, this subspace is the same as the image of $H^2_\ell(\bar E_1^{(2)})(1)\hookrightarrow H^2_\ell(\bar X_1)(1)$, which is an $\Fr$-equivariant injection, and $H^2_\ell(\bar E_1)(1)$ is generated by the classes of a fiber and a section of the Abel-Jacobi map $\bar E_1^{(2)}\to \bar E_1$.  Hence we may write
\[
P_2(t) = (1-pt)^2\cdot \prod_{i=3}^{12}(1-\alpha_i t) = (1-pt)^2Q(t),
\]
where we note that $Q(t)\in\Z[t]$.

Now we will consider the unramified double cover $\bar Y_1$ of $\bar X_1$.  Looking first in characteristic zero, we note that $e(Y_1)=2e(X_1)=20$, $K_{Y_1}^2=2K_{X_1}^2=4$, and thus by Noether's formula $\chi(Y_1)=2\chi(X_1)=2$.  Moreover, from Lemma \ref{alb_lemma}, we have $q(Y_1)=1$ and thus $h^{0,2}(Y_1)=h^{0,2}(Y_1)=2, h^{1,1}(Y_1)=18$.  Since $\tilde E_1$ is isogenous to $E_1$, we have the following $\Gal(\bar \F_p/\F_p)$-module isomorphisms:
\[
H^1_\ell(\bar Y_1) \tilde\To H^1_\ell(\bar{\tilde E}_1) \tilde\To H^1_\ell(\bar E_1) \tilde\To H^1_\ell(\bar X_1).
\]
This implies that
\[
Z(\bar Y_1,t) = \frac{P_1(t)P_3(t)}{P_0(t)\tilde P_2(t)P_4(t)},
\]
where $P_0(t),P_1(t),P_3(t),P_4(t)$ are the polynomials appearing in $Z(\bar X_1,t)$ and where $\tilde P_2(t)$ is a polynomial of degree 22.

But we can say more about $\tilde P_2(t)$.  Looking at middle cohomology, the fact that $X_1$ is the quotient of $Y_1$ by the free involution $\iota$ (which is defined over $\Q$) implies that we may identify the $\Gal(\bar \F_p/\F_p)$-module $H^2_\ell(\bar X_1)$ with the submodule $[H^2_\ell(\bar X_1)]^{\iota}$ of $\iota^\ast$-invariant classes in $H^2_\ell(\bar X_1)$; see \cite[Prop.\ 6.8]{DelMil}. Since
\begin{equation}
P_2(t) = \det\big(1-t\cdot \Fr^\ast \vert {H^2_\ell(\bar X_1)} \big), \qquad \tilde P_2(t) = \det\big(1-t\cdot \Fr^\ast \vert {H^2_\ell(\bar Y_1)} \big),
\end{equation}
it follows that 
\begin{equation}\label{tilde_p2}
\tilde P_2(t) = P_2(t)R(t) = (1-pt)^2Q(t)R(t)
\end{equation}
for a polynomial $R(t)\in \Z[t]$ of degree 10.  The following proposition summarizes the relation between the zeta functions of $\bar X_1$ and $\bar Y_1$:

\begin{proposition}\label{zeta_relation}
Let $E_1, \tilde E_1$, and a $2$-isogeny $\varphi\colon \tilde E_1\to E$ all be defined over $\Q$, and select a point $a_1\in \PP^6(\Q)$.  As in Definition \ref{Psi_defn}, suppose the divisor $B(a_1)\subseteq \symEone$ has only rational double points, let $X_1:=X(a_1)$ be the associated surface with $p_g=q=1$ and $K^2=2$, and let $Y_1:=\tilde X(a_1)$ be the corresponding double cover.  Suppose furthermore that $X_1$ and $Y_1$ have good reductions $\bar X_1$ and $\bar Y_1$ to $\F_p$, respectively, for some prime $p$.  Then the zeta functions of $\bar X_1$ and $\bar Y_1$ over $\F_p$ are related as follows:
\[
Z(\bar Y_1,t) = \frac{Z(\bar X_1,t)}{R(t)}
\]
where $R(t)\in\Z[t]$ is a polynomial of degree 10 whose inverse roots all have complex absolute value $p$.
\end{proposition}

\subsection{An example.}\label{X1_subsec}  Keeping the notation from the previous section, we make the following choices.  Let $\tilde E_1$ be the elliptic curve
\[
y^2 = (x-1)(x^2+1),
\]
with 2-torsion point $C=(1,0)$.  Thus $E_1=\tilde E_1/\langle C \rangle$ and $\varphi\colon \tilde E_1\to E_1$ is the quotient map, and both of these are defined over $\Q$ as well; explicitly, although we will not need to make use them, we may write them in equations as
\begin{equation}\label{E1_eqn}
E_1\colon \quad Y^2 = X^3-X^2-9X-7
\end{equation}
and
\[
\varphi(x,y) = \left(\frac{x^2 - x + 2}{x - 1}, \frac{y(x^2 - 2 x- 1)}{(x-1)^2}\right).
\]
Let
\[
a_1 := (1 : 0 : 0 : 1 : 0 : -1 : 1) \in \PP^6(\Q),
\]
let $X_1 = X(a_1)$, and $Y_1 = \tilde X(a_1)$.  Finally, let $\bar X_1$ and $\bar Y_1$ be the reduction of these surfaces to the finite field $\F_3$.

\begin{proposition}\label{B1_smooth}
The curve $B_1 = B(a_1)$ in $\symEone$ is smooth and has good reduction to $\F_3$, and thus the canonical divisor of the surface $X_1$ is ample.
\end{proposition}

\begin{proof}
It suffices to show that the singular locus of the unramified cover $\tilde B(a_1)$ in $\tilde P$ is empty, both over $\bar \Q$ and $\bar \F_3$, and for this one may use the equations from \S\ref{double_sec} in a computer algebra package such as \textsc{Singular}.  See \cite{LyData} for computational details.  

The second statement then follows from Theorem \ref{class}.
\end{proof}

From this proposition, we see that $\bar X_1$ and $\bar Y_1$ are both smooth projective surfaces over $\F_3$.  We now consider the zeta function $Z(\bar Y_1,t)$ of $\bar Y_1$ over $\F_3$.  One checks that the elliptic curve $\tilde E_1$ has good reduction to $\F_3$ and that $\#\tilde E_1(\F_3) = 6$.  Thus, by the discussion in the previous subsection, we have
\[
Z(\bar Y_1,t) = \frac{(1-2t+3t^2)(1-6t+27t^2)}{(1-t)\tilde P_2(t) (1-9t)}.
\]
Here $\tilde P_2(t) = (1-3t)^2 S(t)$ is a polynomial $S(t)\in \Z[t]$ of degree 20 whose inverse roots $\beta_i$ all have complex absolute value $3$.  Moreover, by Poincar\'e duality on $H^\ast_\ell(\bar Y_1)$, one knows that the map $\beta_i\mapsto \frac{9}{\beta_i}$ is a permutation of the inverse roots $\beta_i$.  It follows from this that
\[
t^{20} S(1/t) = \eps \cdot 3^{20} S(t/9),
\]
where $\eps=\pm1$.  (Specifically, $\eps =-1$ if and only if $-3$ occurs with odd multiplicity among the $\beta_i$.)  Writing $S(t) = \sum_{k=0}^{20} b_k t^k$, it follows that
\begin{equation}\label{func_eqn}
b_{10+j} = \eps\cdot 9^{j}\cdot b_{10-j}
\end{equation}
In general, this means that $\eps=9b_9/b_{11}$ and thus $S(t)$ is determined by $b_1, b_2, \ldots, b_{11}$.  In special cases, though, one may can infer $\eps$ (and thus $S(t)$) from just $b_1, b_2, \ldots, b_{10}$.  For instance, by (\ref{func_eqn}) one concludes that $\eps =1$ whenever $b_{10}\neq 0$.  When $b_{10}=0$, one may have $\eps=\pm 1$, but one of the two possibilities can possibly be ruled out by using (\ref{tilde_p2}), since $S(t)$ must have a degree 10 factor in $\Q[t]$.

On the other hand, the Lefschetz fixed point formula applied to the automorphism $\Fr$ acting upon $\bar Y_1$ allows one to write
\[
Z(\bar Y_1,t) = \exp \left( \sum_{n=1}^\infty \big(\# \bar Y_1(\F_{3^n}) \big) \frac{t^n}{n} \right),
\]
and from this we obtain
\begin{equation}\label{point_count}
\# \bar Y_1(\F_{3^n}) = \left(\sum_{i=1}^{20} \beta_i^n\right) +(1+3^n)^2-(1+3^n)(\gamma_1^n+\gamma_2^n),
\end{equation}
where $S(t) = \prod_{i=1}^{20}(1-\beta_i t)$ and where $\gamma_1,\gamma_2=1\pm \sqrt{-2}$ are the inverse roots of $P_1(t) = 1-2t+3t^2$.

Let us also note the connection between the coefficients $b_k$ of $S(t)$ and powers of its inverse roots $\beta_i$:  Defining $s_n := \sum_{i=0}^{20} \beta_i^n$, Newton's identities give
\begin{equation}\label{newton}
b_k = -\frac1k\sum_{j=1}^k b_{k-j} s_j.
\end{equation}
Putting this all together, we may calculate the following:

\begin{proposition}\label{zeta_Y}
The zeta function of $\bar Y_1$ over $\F_3$ is
\[
Z(\bar Y_1,t) = \frac{(1-2t+3t^2)(1-6t+27t^2)}{(1-t)\tilde P_2(t) (1-9t)},
\]
where $\tilde P_2(t)$ factors into irreducibles in $\Z[t]$ as
\begin{eqnarray*}
\tilde P_2(t) &=& (1 - 3 t)^2 \cdot (1 + 3 t)^2 \cdot (1 - t + 12 t^2 + 108 t^4 + 972 t^6 -   729 t^7 + 6561 t^8) \cdot  \\
&& (1 - t + 9 t^2 - 45 t^3 + 108 t^4 - 324 t^5 +  972 t^6 - 3645 t^7 + 6561 t^8 - 6561 t^9 + 59049 t^{10})
\end{eqnarray*}
\end{proposition}

\begin{proof}
By the discussion above, one may use the identities (\ref{func_eqn}), (\ref{point_count}), and (\ref{newton}) to determine $S(t)$ completely if one knows the point counts $\#\bar Y_1(\F_{3^n})$ for small values of $n$, and this may be done in \textsc{Magma}.  Specifically, one uses the observation that every element of $\bar Y_1(\F_{3^n})$ lies above exactly one point of $\overline{\tilde E}_1(\F_{3^n})$ via the Albanese fibration $Y_1\to \tilde E_1$.  Thus one has the disjoint union
\[
\bar Y_1(\F_{3^n}) = \coprod_{Q\in \overline{\tilde E}_1( \F_{3^n})} \bar Y_{1,Q}(\F_{3^n}),
\]
where a general fiber $\bar Y_{1,Q}$ is a smooth curve over $\F_{3^n}$ of genus two.  Hence one may first use \textsc{Magma} to generate a list of all points in $\overline{\tilde E}(\F_{3^n})$ and then, using the equations (\ref{away_fibers}) and (\ref{O_fiber}), determine $\#\bar Y_{1,Q}(\F_{3^n})$ for each $Q\in \overline{\tilde E}( \F_{3^n})$.

Putting $S(t)=\sum_{k=0}^{20} b_k t^k$ as above, one may use \textsc{Magma} to determine $\#\bar Y_1(\F_{3^n})$ for $n=1,2,\ldots,10$ and then use (\ref{point_count}) and (\ref{newton}) to obtain the coefficients $b_1,\ldots, b_{10}$.  It turns out that $b_{10}=0$, but upon applying (\ref{func_eqn}) with $\eps=-1$, one obtains an answer for $S(t)$ that does not contain a factor of degree $10$ in $\Q[t]$, violating (\ref{tilde_p2}). Thus we must have $\eps=1$, and applying (\ref{func_eqn}) with this value leads to the polynomial $\tilde P_2(t)=(1-3t)^2S(t)$ as in the statement of the proposition.

See \cite{LyData} for \textsc{Magma} code pertaining to this calculation.
\end{proof}

\begin{theorem}
The zeta function of $\bar X_1$ over $\F_3$ is 
\[
Z_1(\bar X_1,t) = \frac{(1-2t+3t^2)(1-6t+27t^2)}{(1-t)P_2(t) (1-9t)},
\]
where
\[
P_2(t) = (1 - 3 t)^2 (1 - t + 9 t^2 - 45 t^3 + 108 t^4 - 324 t^5 +  972 t^6 - 3645 t^7 + 6561 t^8 - 6561 t^9 + 59049 t^{10}).
\]
\end{theorem}

\begin{proof}
By Propositions \ref{zeta_relation} and \ref{zeta_Y}, as well as the identity (\ref{tilde_p2}), we have
\[
Z_1(\bar X_1,t) = \frac{(1-2t+3t^2)(1-6t+27t^2)}{(1-t)P_2(t) (1-9t)},
\]
where either
\[
P_2(t) = (1 - 3 t)^2 (1 - t + 9 t^2 - 45 t^3 + 108 t^4 - 324 t^5 +  972 t^6 - 3645 t^7 + 6561 t^8 - 6561 t^9 + 59049 t^{10})
\]
or
\[
P_2(t) = (1 - 3 t)^2 (1 + 3 t)^2 (1 - t + 12 t^2 + 108 t^4 + 972 t^6 -   729 t^7 + 6561 t^8).
\]
To differentiate between the two possibilities, we may use the connection between $Z_1(\bar X_1,t)$ and point counting on $\bar X_1$.  Specifically, the first possibility predicts that $\#\bar X_1(\F_3) = 9$ and the second possibility predicts that $\#\bar X_1(\F_3)=3$.

To decide which case holds, we may count points on $\bar X_1$ by using the fact that $\bar X_1$ is the quotient of $\bar Y_1$. One reference for basic facts about quotients of varieties by finite groups is \cite[Expos\'e 5]{Groth-SGA1}.  To start with, the fibers of the quotient map are precisely the $\iota$-orbits on $\bar X_1$, so we may write
\[
\bar X_1(\bar \F_{3}) = \bar Y_1(\bar \F_3)/\iota.
\]
Moreover, as $\iota$ and the quotient map are defined over $\F_3$, we may say that
\[
\bar X_1(\F_{3^n}) = \bar X_1(\bar \F_3)^{\Fr^n} = \big(\bar Y_1(\bar \F_3)/\iota\big)^{\Fr^n},
\]
i.e., the $\F_{3^n}$-points of $\bar Y_1$ correspond precisely to the $\Fr^n$-invariant $\iota$-orbits of the $\bar \F_3$-points of $\bar X_1$. To finish this circle of ideas, we note that any $\iota$-orbit consists of two points in $\bar Y_1(\bar \F_3)$, and hence if such an orbit is $\Fr^n$-invariant then the points within this orbit must be $\Fr^{2n}$-invariant.  Thus we may write 
\[
\bar X_1(\F_{3^n}) = \big(\bar Y_1(\F_{3^{2n}})/\iota\big)^{\Fr^n},
\]
and in particular we have
\[
\#\bar X_1(\F_{3}) = \# \big(\bar Y_1(\F_{9})/\iota\big)^{\Fr}.
\]

The calculation of $\Fr$-invariant $\iota$-orbits of $\bar Y_1(\F_9)$ is straightforward, but a few observations can cut down the work (to the point that it may even be done relatively quickly by hand).  Recall that $\iota$ is the lift to $\tilde P$ of the involution $Q\mapsto Q+C$ on $\tilde E_1$.  Over $\tilde E_1\setminus\set{O,C}$, we may denote points using the relative projective coordinates $(Q; Z_0: Z_1)$, where $Q\in\tilde E_1$, and by (\ref{involution}) an $\iota$-orbit in this open subset of $\tilde P$ looks like
\[
\set{ (Q; Z_0: Z_1), (Q+C; Z_1: Z_0)}.
\]
Hence if such an orbit in $\bar Y_1(\F_{9})/\iota$ is $\Fr$-invariant, then either both points belong to $\bar Y_1(\F_3)$ or we have $\Fr(Q)=Q+C$ and $(Z_0^3:Z_1^3) = (Z_1:Z_0)$.  For $Q\in \overline{\tilde E}_1(\F_9)\setminus \{\tilde O,C\}$, this observation narrows down the fibers $\bar Y_{1,Q}$ that we must examine, as well as the types of points on those fibers; one finds $7$ points in $\bar X_1(\F_3)$ in this way.

Finally, we must consider orbits involving points in the fibers $\bar Y_{1,\tilde O}$ and $\bar Y_{1,C}$.  Given a point $\sigma_1\in \bar Y_{1,\tilde O}(\F_9)$, its $\iota$-orbit will have the form $\set{\sigma_1,\sigma_2}$, where $\sigma_2\in \bar Y_{1,C}(\F_9)$.  But since $\Fr$ maps $\bar Y_{1,\tilde O}$ and $\bar Y_{1,C}$ to themselves, the orbit $\set{\sigma_1,\sigma_2}$ is $\Fr$-invariant if and only if $\Fr(\sigma_1)=\sigma_1$.  In other words, the $\Fr$-invariant $\iota$-orbits arising from points in $\bar Y_{1,\tilde O}(\F_9)$ and $\bar Y_{1,C}(\F_9)$ are in bijection with the points of $\bar Y_{1,\tilde O}(\F_3)$; using (\ref{O_fiber}), we find that this gives 2 additional points in $\bar X_1(\F_3)$.

Hence we conclude that $\#\bar X_1(\F_3)=9$, so that the first case holds.
\end{proof}

Theorem \ref{thmA} now follows from:

\begin{corollary}\label{pic_equals_2}
The variety $X_1$, as defined by the choices in this section, is a surface over $\Q$ with $p_g=q=1$ and $K^2=2$ whose Picard number is $\rho(X_1)=2$.
\end{corollary}

\begin{proof}
By (\ref{pic_bound}), we may conclude that $\rho(\bar X_1)=2$ if the polynomial
\[
P_2(t) = (1 - 3 t)^2 (1 - t + 9 t^2 - 45 t^3 + 108 t^4 - 324 t^5 +  972 t^6 - 3645 t^7 + 6561 t^8 - 6561 t^9 + 59049 t^{10})
\]
has only two inverse roots of the form $3\cdot$(root of unity); this is easily seen by noting that $P_2(t/3)$ decomposes in $\Q[t]$ as
\[
P_2(t/3) = (1 - t)^2 \left( 1 - \frac{t}3 + t^2 - \frac{5}3 t^3 + \frac{4}{3} t^4 - \frac43 t^5 + \frac43 t^6 - \frac53 t^7 + t^8 - \frac13 t^9 + t^{10} \right),
\]
and the irreducible degree 10 factor is clearly not a cyclotomic polynomial.

To finish, we only need to note that the geometric Picard number of a smooth projective variety over $\Q$ is bounded above by the geometric Picard number of its reduction to $\F_p$, for any prime $p$ of good reduction.
\end{proof}


\section{Singular Elements in $\Dsys$}\label{sing_sec}

Let $E$ be any elliptic curve over $k$ and consider the linear system $\Dsys$ on $\symE$.  Since $\Dsys$ is base point free (see Proposition \ref{not_very_ample}), Bertini's Theorem implies that a general element of $\Dsys$ is nonsingular.  Let $R\subseteq \Dsys$ denote the subset of singular elements.  If one chooses an identification of $\Dsys$ with the projective space $\PP^6$, then $R$ may be endowed with the structure of a projective subvariety of $\PP^6$.

\begin{lemma}\label{codim_one}
The subvariety $R$ has codimension one in $\Dsys$.
\end{lemma}

\begin{proof}
Choose a general pencil of divisors in $\Dsys$ and let $T\to \PP^1$ denote the total space of this pencil.  As $\Dsys$ is base point free, it follows (by taking a general element $B\in \Dsys$ and applying Bertini's theorem to the trace of $\Dsys$ on $B$) that a general pencil in $\Dsys$ has a smooth base locus $\Gamma$ consisting of $\mathscr D^2=12$ points in $\symE$.  Note that $T$ is the blow-up of $\symE$ along $\Gamma$.  If the lemma were false, then $T(\C)\to \PP^1(\C)$ would be a topological fibration and we would have $e(T) = e(\PP^1)e(B)$, where $B$ is a (smooth) element of the pencil.  However, one checks directly that
\[
e(T) - e(\PP^1)e(B) = 12-2\cdot (-10) = 32 \neq 0.
\]
Hence a general pencil in $\Dsys$ must contain singular elements, showing that $R$ has codimension one in $\Dsys$.
\end{proof}

More can be said about $R$ if one knows the existence of a certain kind of pencil in $\Dsys$:

\begin{proposition}\label{theoretical}
Assume that there exists a linear system $\mathfrak d\subseteq \Dsys$ of (projective) dimension $3$ and a pencil $\Pi\subseteq \mathfrak d$ with the following properties:
\begin{enumerate}
\item\label{32sing} The pencil $\Pi$ contains at least $32$ singular elements and its general element is smooth.
\item\label{isolated} Any singular element in $\Pi$ contains only isolated singularities. 
\item\label{rat_map} At least $32$ singular elements $B\in \Pi$ have the following property: There exists $q\in B\subseteq \symE$ such that $B$ is the only element of $\mathfrak d$ containing $q$ as a singular point.
\end{enumerate}
Then $R$ has only one irreducible component $R_0$ of codimension one in $\Dsys$, which is an irreducible hypersurface of degree $32$, and the singular locus of a general element in $R_0$ is one ordinary double point. 
\end{proposition}

\begin{proof}
By assumptions (\ref{32sing}) and (\ref{rat_map}), there is a rational map $\eta:\symE \dashrightarrow \mathfrak d\simeq \PP^3$ that sends a general point $q\in\symE$ to the unique divisor $B\in\mathfrak d$ that is singular at $q$.  The Zariski-closure of the image of $\eta$ defines an irreducible hypersurface $\Sigma \subseteq \mathfrak d$ such that  $\deg(\Sigma)\geq 32$.  From this it follows that $R$ contains an irreducible component $R_0$ of codimension one in $\Dsys$ such that $\deg(R_0)\geq 32$ and $R_0\cap \mathfrak d = \Sigma$.  Moreover, as all elements of $\Pi\cap R$ have isolated singularities by assumption (\ref{isolated}), the same is true for a general element of any codimension one component of $R$.

Now fix a general pencil $\Pi'\subseteq \Dsys$.  Since $\Pi'$ lies in general position with respect to $R$, it only intersects the codimension one components of $R$, it  has at least 32 singular elements (since $\deg(R_0)\geq 32$), and all of its singular elements have only isolated singularities.  Moreover, arguing as in the proof of Lemma \ref{codim_one}, we may conclude that the base locus of $\Pi'$ is smooth.  Letting $T\to \Pi'\simeq \PP^1$ denote the total space, we deduce that $T$ is smooth and that, if $CP\subseteq T$ denotes the set of critical points and $CV\subseteq \Pi'$ the set of critical values of this map, then
\[
32\leq |CV| \leq |CP| <\infty.
\]
Now for a critical point $q\in CP$ lying over the critical value $B\in CV$, let $\mu_B(q)$ denote the Milnor number of the isolated singularity $q\in B$.  In particular, we have $\mu_B(q)\geq 1$, with equality if and only if $q$ is an ordinary double point of $B$.  But upon applying \cite[Example 14.1.5(d)]{Ful}, we find that
\[
\sum_{q\in CP} \mu_{B}(q) = e(T) - e(\PP^1)e(B) = 12-2\cdot (-10) = 32.
\]
The conclusion is that $\Pi'$ contains exactly 32 singular elements and that the singular locus of each of these is one ordinary double point.  Moreover, as $\deg(R_0)\geq 32$, it follows that $\deg(R_0)=32$ and that $R_0$ is the only irreducible component of codimension one in $R$.  Finally, from the description of the singular elements of $\Pi'$, we find that the singular locus of a general element of $R_0$ will be one ordinary double point.
\end{proof}

\begin{proposition}\label{pencil_exists}
Let $E_1$ be the elliptic curve defined in (\ref{E1_eqn}), let $\mathfrak d$ be the projective subspace of $\Dsys$ on $\symEone$ spanned by (the zero loci of) the four sections $\Psi_0,\Psi_1, s_1,s_2\in H^0(\symEone,\mathcal O_{\symEone}(\mathscr D))$, where
\begin{eqnarray*}
s_1 &:=& \Psi_0 +\Psi_3 - \Psi_5 + \Psi_6 \\
s_2 &:=& \Psi_0 + \Psi_1 + \Psi_2,
\end{eqnarray*}
and let $\Pi\subseteq \mathfrak d$ denote the pencil spanned by $s_1,s_2$.  Then these choices $E_1, \mathfrak d, \Pi$ satisfy assumptions (\ref{32sing}), (\ref{isolated}), (\ref{rat_map}) of Proposition \ref{theoretical}.

Consequently, for the surface $\symEone$, the collection $R\subseteq \Dsys$ of singular elements has only one irreducible component of codimension one, and a general element $B\in R$ has a singular locus of one ordinary double point.
\end{proposition}

\begin{proof}
Recall that $E_1 = \tilde E_1/\langle C \rangle$, where $\tilde E_1$ is the elliptic curve
\[
\tilde E_1\colon \quad y^2 = (x-1)(x^2+1)
\]
and $C=(1,0)$.  Pulling back via the double cover $\Phi: \tilde P\to \symEone$, the linear system $\Phi^\ast \mathfrak d$ is spanned by $\psi_0, \psi_1, \tilde s_1,$ and $\tilde s_2$, where
\begin{eqnarray*}
\tilde s_1 &:=& \psi_0 +\psi_3 - \psi_5 + \psi_6 \\
\tilde s_2 &:=& \psi_0 + \psi_1 + \psi_2,
\end{eqnarray*}
and the pencil $\Phi^\ast \Pi$ is spanned by $\tilde s_1$ and $\tilde s_2$.  To prove the proposition, we will show
\begin{enumerate}
\item[\ref{32sing}'] The pencil $\Phi^\ast \Pi$ contains at least $32$ singular elements.
\item[\ref{isolated}'] Each singular element in $\Phi^\ast \Pi$ contains only isolated singularities.
\item[\ref{rat_map}'] At least $32$ singular elements $\tilde B\in \Phi^\ast\Pi$ have the following property: There exists $\tilde q\in \tilde B\subseteq \symE$ such that $\tilde B$ is the only element of $\Phi^\ast\mathfrak d$ containing $\tilde q$ as a singular point.
\end{enumerate}
This is readily accomplished using the equations in \S\ref{double_sec} and with the help of \textsc{Singular}.  If $U=\tilde P \setminus\set{\tilde P_{\tilde 0}, \tilde P_C,Z(Z_1)}$ and $\Omega = \Phi^\ast \Pi\setminus\set{\tilde s_1}$, then we consider the affine open subset $U\times \Omega \subseteq \tilde P\times \Phi^\ast\Pi$, on which we have convenient coordinates and convenient equations for $\tilde T\cap (U\times \Omega)$, where $\tilde T$ is the total space of $\Phi^\ast \Pi$.  We may first form the ideal defining the subscheme of points in $\tilde T\cap (U\times \Omega)$ that arise as singularities of some element of $\Omega$.  This is an ideal $I$ defined by polynomials with coefficients in $\Q$, and \textsc{Singular} shows that its reduction $I_7$ to $\F_7$ has Krull dimension 0 and $\F_7$-dimension 64.  It follows that $I$ also has Krull dimension 0 and $\Q$-dimension \emph{at least} 64.\footnote{In fact \textsc{Singular} does a reasonable job of handling the Gr\"{o}bner basis calculations for the ideal $I$ entirely in characteristic zero, but the corresponding calculations obtained by reducing modulo a small prime are still much faster.  More importantly, though, the Gr\"{o}bner basis computations further into the proof are too demanding in characteristic zero, so passage to a finite field does seem necessary at some point.}  Moreover, \textsc{Singular} computes that the intersection of $I_7$ with the univariate polynomial ring $\F_7[\Omega]$ is generated by a polynomial of degree 32.  Geometrically, this implies (by appealing to the symmetry of $\tilde P$ under the involution $\iota$) that, in characteristic zero, there are \emph{at least} 32 singular elements in $\Omega$, each possessing 2 singular points in $U$.  After noting Proposition \ref{B1_smooth} gives the smoothness of the element $Z(s_1)=B_1$, items \ref{32sing}' and \ref{isolated}' follow.  

To verify \ref{rat_map}', for each of the points in $Z(I)$ we may examine the subspace of $\Phi^\ast \mathfrak d$ consisting of elements that pass through the point and have a singularity there.  Upon verifying that no point $\tilde q\in Z(I)$ lies over a 2-torsion point of $\tilde E_1$ (by showing in \textsc{Singular} that $Z(I_7,y)=\emptyset$ in characteristic 7), it follows that the functions $x-x(\tilde q)$ and $Z_0/Z_1-(Z_0/Z_1)(\tilde q)$ give local parameters of $\tilde P$ at $\tilde q$.  Thus we form the matrix
\[
\begin{bmatrix}
\psi_0 & \psi_1 & \tilde s_1 & \tilde s_2 \\
{\partial \psi_0}/{\partial x} & {\partial \psi_1}/{\partial x} & {\partial \tilde s_1}/{\partial x} & {\partial \tilde s_2}/{\partial x} \\
{\partial \psi_0}/{\partial Z_0} & {\partial \psi_1}/{\partial Z_0} & {\partial \tilde s_1}/{\partial Z_0} & {\partial \tilde s_2}/{\partial Z_0} 
\end{bmatrix},
\]
and the projectivized kernel of this matrix at a given point $\tilde q$ gives all elements of $\Phi^\ast \mathfrak d$ containing $\tilde q$ in its singular locus. Letting $J$ be the ideal of its $3\times 3$ minors, and $J_7$ its reduction to $\F_7$, we check in \textsc{Singular} that $Z(I_7,J_7) = \emptyset$; hence this matrix has full rank at each point of $Z(I)$ and \ref{rat_map}' follows.
\end{proof}

\begin{remark}
Let us connect the statement in Proposition \ref{pencil_exists} to the discussion following Theorem \ref{thmB} in \S\ref{intro}.  Recall from the proof of Corollary \ref{family} that the family $\mathcal X\to S$ is defined as the collection of smooth fibers within the larger flat family $\mathcal X_0\to\mathcal S_0$.  By construction, $\mathcal S_0$ is a projective bundle over the open modular curve $A$ such that, if $\alpha \in A(k)$ corresponds to the elliptic curve $E$, the fiber $\mathcal S_{0,\alpha}$ becomes identified with the linear system $\Dsys$ on $\symE$.  If $V = \mathcal S_0\setminus \mathcal S$ denotes the locus of singular fibers in $\mathcal X_0\to\mathcal S_0$, then $V\cap \mathcal S_{0,\alpha}$ has codimension one inside $S_{0,\alpha}$ by Lemma \ref{codim_one}.  By Proposition \ref{pencil_exists}, the codimension one components of $V\cap \mathcal S_{0,\alpha}$ form a hypersurface of degree 32 in $\mathcal S_{0,\alpha}$, and for all but perhaps finitely many $\alpha$ this hypersurface is irreducible.  Finally, the same proposition implies that a general element of $V$ has a singular locus of exactly one ordinary double point.

In the next section, we will consider subfamilies of $\mathcal X_0\to\mathcal S_0$ obtained from general pencils inside the linear system $\Dsys$ on $\symEone$; these are the ``one-parameter families'' alluded to in \S1.  While we will only need to work with pencils coming from the particular surface $\symEone$, we note by the previous paragraph that $E_1$ can be replaced with almost any other elliptic curve $E$.
\end{remark}


\section{Picard-Lefschetz Theory and Big Monodromy}\label{PL_mon_sec}

\subsection{}\label{PL_subsec} In this section we apply Picard-Lefschetz theory to certain pencils of surfaces with $p_g=q=1$ and $K^2=2$, in order to deduce Theorem \ref{thmB}.  We pay particular attention to the determination of the spaces of monodromy invariants and vanishing cycles in such pencils, with Theorem \ref{thmA} playing a key role (see the proof of Proposition \ref{dim_I}).

Let $E_1$ denote the elliptic curve defined in (\ref{E1_eqn}) and consider the divisor let $B_1:=B(a_1)\in \Dsys$ on $\symEone$ from \S\ref{X1_subsec}; note that in Proposition \ref{pencil_exists} this was written in the alternate notation $B_1=Z(s_1)$.  By definition, the double cover of $X_1\to \symEone$ from \S\ref{X1_subsec} is branched over $B_1$.  By Proposition \ref{B1_smooth}, $B_1$ is smooth (and hence $X_1$ has ample canonical divisor) and by Corollary \ref{pic_equals_2} we have $\rho(X_1)=2$.  Now fix a general pencil $\Pi\in \Dsys$ that passes through the element $B_1$.  By Proposition \ref{pencil_exists}, $\Pi$ contains exactly 32 singular elements, and the singular locus of each such element is one ordinary double point.  Moreover, arguing as in Lemma \ref{codim_one}, the base locus of $\Pi$ is smooth.  If we let $T\subseteq \symEone\times\Pi$ denote the total space of $\Pi$, it follows that $T$ is smooth and that
\[
\mathcal O_{\symEone\times\Pi}(T)\simeq p_1^\ast\mathcal O_{\symEone}(\mathscr D),
\]
where $p_1:\symEone \times\Pi\to\symEone$ is the first projection.  Let $Y$ be the double cover of $\symEone\times \Pi$ that is branched over $T$ (inside the line bundle associated to the pullback of $\mathcal O_{\symEone}(3D_0-G_0)$ to $\symEone\times\Pi$).  Then $Y\to \Pi\simeq \PP^1$ is a fibration into surfaces; each fiber is the canonical model of a surface with $p_g=q=1$ and $K^2=2$, and the general element is smooth.  Since one of the fibers is isomorphic to the surface $X_1$, we will denote by $b\in \Pi$ a base point such that $Y_b\simeq X_1$.  Finally, there are exactly 32 singular fibers in $Y\to \Pi$, and the singular locus of each one is one ordinary double point. 

It follows that one may apply Picard-Lefschetz theory to the fibration $Y\to\Pi\simeq \PP^1$.  We now describe the relevant features that result from this theory, and refer to \S5 and \S6 of \cite{Lamot} for more details.  Picard-Lefschetz theory defines a subspace of \emph{vanishing cycles} $V\subseteq H_2(Y_b,\Q)$ as follows.  Let $\Pi_{sm}\subseteq\Pi$ denote the locus of smooth fibers, let $D\subseteq \Pi_{sm}$ denote a small open disc such that $b\in\partial D$, and restrict $Y\to\Pi$ to the subfamily $Y_+$ that lies over $\Pi\setminus D$.  Then using the embedding $Y_b\hookrightarrow Y_+$, one defines
\begin{equation}\label{van_def}
V := \ker \left(H_2(Y_b,\Q) \to H_2(Y_+,\Q)\right).
\end{equation}
Let $V^\ast\subseteq H^2(Y_b,\Q(1))$ denote image of $V$ under the isomorphisms
\begin{equation}\label{homology_comp}
H_2(Y_b,\Q) \stackrel{\text{P.D.}}{\To} H^2(Y_b,\Q) \stackrel{\otimes \Q(1)}{\To} H^2(Y_b,\Q(1)),
\end{equation}
where the first map denotes Poincar\'e duality.  The theory furthermore associates to each singular element of $\Pi$ a distinguished nonzero element $\delta_i\in V^\ast$ ($i=1,\ldots,32$) and shows that the collection of all $\delta_i$ generate $V^\ast$.

As $\Pi_{sm}$ is the complement of 32 points in $\Pi\simeq \PP^1$, the fundamental group $\pi_1(\Pi_{sm}(\C),b)$ is generated by the homotopy classes $\gamma_i$ of 32 nonintersecting elementary paths around these points.  (In other words, we choose $\gamma_i$ to have winding number one around the $i$th puncture and winding number zero around the $j$th puncture for $i\neq j$.)  Let $\psi_b$ denote the nondegenerate symmetric bilinear form on $H^2(Y_b,\Q(1))$ coming from the cup product and let
\begin{equation}\label{orig_pencil_mon}
\rho: \pi_1(\Pi_{sm}(\C),b)\to \orth(H^2(Y_b,\Q(1)),\psi_b)
\end{equation}
denote the monodromy representation.  Then the image of $\rho$ is generated by the collection of elements $T_i = \rho(\gamma_i)$.  The \emph{Picard-Lefschetz formula} states:
\begin{equation}\label{PL_form}
T_i(x) = x+\psi_b(x,\delta_i)\delta_i.
\end{equation}
Furthermore, one also has:
\begin{equation} \label{delta_length}
\psi_b(\delta_i,\delta_i)=-2.
\end{equation}

A second subspace of $H^2(Y_b,\Q(1))$ is
\[
I = H^2(Y_b,\Q(1))^{\pi_1(\Pi_{sm}(\C),b)},
\]
the subspace of classes invariant under the monodromy action $\rho$.  One knows that $\dim I\geq 2$, by considering the pullbacks of the divisors $D_0\times \Pi$ and $G_0\times \Pi$  from $\symEone\times \Pi$ to $Y$.  Moreover, it follows from (\ref{PL_form}) that
\begin{equation}\label{V_perp}
I = (V^\ast)^\perp.
\end{equation}

\begin{proposition}\label{dim_I}
One has $\dim I = 2$.  In fact, letting $c\colon Y_b\to \symEone$ denote the double covering branched over $B_1$, we have
\[
I = \IM\left( H^2(\symEone,\Q(1))\stackrel{c^\ast}{\to} H^2(Y_b,\Q(1))\right).
\]
\end{proposition}

\begin{proof}
As the canonical class $K$ on $Y_b$ is ample, let $P^2(Y_b):=[K]^\perp\subseteq H^2(Y_b,\Q(1))$ denote the primitive cohomology of $Y_b$ with respect to the polarization $[K]$.  Then the restriction of $\psi_b$ to $P^2(Y_b)$ is nondegenerate and makes $P^2(Y_b)$ into a polarized Hodge structure with Hodge numbers $h^{1,1}=9, h^{2,0}=h^{0,2}=1$.  Let $I_P = I\cap P^2(Y_b)$.  Since $[K]\in I$, we have $\dim I_P\geq 1$ and it suffices to show $\dim I_P=1$ to yield the first statement.

Note that $V^\ast\subseteq P^2(Y_b)$ by (\ref{V_perp}), and it is a nonzero subspace by (\ref{delta_length}).  Moreover, the orthogonal complement of $V^\ast$ in $P^2(Y_b)$ is $I_P$.  This implies that $I_P\neq P^2(Y_b)$.  Finally, letting $I_P^\perp$ denote the orthogonal complement of $I_P$ in $P^2(Y_b)$, we also have $I_P^\perp\neq 0$; indeed, it contains the nonzero subspace $\Q\delta_i\subseteq I_P^\perp$ by (\ref{PL_form}).

Using Deligne's Theorem of the Fixed Part, $I$ is a Hodge substructure of $H^2(Y_b,\Q)(1)$, and hence the same is true of $I_P$ in $P^2(Y_b)$.  But $P^2(Y_b)$ is polarized, so we have a sum of Hodge structures $P_2(Y_B) = I_P\oplus I_P^\perp$ in which both summands are nonzero and proper; furthermore, $I_P$ contains rational $(1,1)$-classes.  The Hodge numbers of $P^2(Y_b)$ imply that exactly one of $I_P$ and $I_P^\perp$ consists \emph{purely} of $(1,1)$-classes, and the fact that $\rho(Y_b)=2$ from Theorem \ref{thmA} then forces one to conclude that $\dim I_P = 1$.

Now the second statement follows from the fact that $I$ contains the independent classes $[K]=c^\ast[D_0]$ and $[F]=c^\ast[G_0]$ (with $F$ an Albanese fiber).
\end{proof}

\begin{proposition}\label{orthog_decomp}
The cup product form $\psi_b$ is nondegenerate on $V^\ast$ and one has an orthogonal decomposition
\[
H^2(Y_b,\Q(1)) = I\oplus V^\ast.
\]
\end{proposition}

\begin{proof}
Define a subspace $W\subseteq H_2(Y_b,\Q(1))$ as
\[
W = \ker\left(H_2(Y_b,\Q)\stackrel{c_\ast}{\To} H_2(\symEone,\Q)\right),
\]
where the map on homology is induced by the double covering map $c:Y_b\to \symEone$.   Letting $W^\ast\subseteq H^2(Y_b,\Q(1))$ denote the image of $W$ under the composition (\ref{homology_comp}), we first claim that there is a direct sum decomposition
\[
H^2(Y_b,\Q(1)) = I\oplus W^\ast.
\]

To see this, note that an equivalent definition of $W^\ast$ is
\[
W^\ast = \ker\left(H^2(Y_b,\Q(1))\stackrel{c_!}{\To} H^2(\symEone,\Q(1))\right),
\]
where $c_!$ is the Gysin map obtained by pre- and post-composing $c_\ast$ with the relevant Poincar\'e duality isomorphisms.  Since $c$ is finite of degree 2, the composition
\[
c_!\circ c^\ast : H^2(\symEone,\Q) \to H^2(\symEone,\Q)
\]
is simply multiplication by 2.  Since $c^\ast$ is injective with image $I$ by Proposition \ref{dim_I}, one has
\[
\dim H^2(Y_b,\Q) = \dim \IM(c_!) + \dim \ker (c_!) = \dim I +\dim W^\ast
\]
and that $I\cap W^\ast =0$.  Thus $H^2(Y_b,\Q) = I\oplus W^\ast$.

Now the double covering $Y_b\to\symEone$ factors as
\[
Y_b \hookrightarrow Y_+ \to \symEone\times (\Pi\setminus D) \to \symEone
\]
(recall that the closed subset $\Pi\setminus D \subseteq \Pi$ is the base space of the subfamily $Y_+$), so by (\ref{van_def}) one has the inclusions $V\subseteq W$ and, by duality, $W^\ast\subseteq V^\ast$.  We conclude that
\[
H^2(Y_b,\Q(1)) = I\oplus W^\ast \subseteq I\oplus V^\ast \subseteq H^2(Y_b,\Q(1)),
\]
showing the claimed direct sum decomposition; moreover, it is an orthogonal sum by (\ref{V_perp}).  It also shows that $V^\ast \cap (V^\ast)^\perp = V^\ast\cap I = 0$, and hence $\psi_b$ is nondegenerate on $V^\ast$.
\end{proof}


\subsection{}\label{mon_subsec} Recall from Corollary \ref{family} the family $f\colon\mathcal X\to \mathcal S$ over $\Q$ containing all surfaces with $p_g=q=1, K^2=2$, and $K$ ample.  By construction, the base space $\mathcal S$ is an open subset of a $\PP^6$-bundle $\mathcal S_0$ over an open modular curve $A$; for a point $\alpha\in A(\Q)$ on the modular curve corresponding to the elliptic curve $E_1$ (with some level structure), the fiber $\mathcal S_{0,\alpha}$ may be identified with the linear system $\Dsys$ on $\symEone$, while the fiber $\mathcal S_\alpha$ is identified with the open subset $\Dsys\setminus R$ of smooth elements in $\Dsys$.  Furthermore, if $s\in \mathcal S$ (lying over $\alpha\in A$) corresponds to the smooth element $B\in \Dsys$, then $\mathcal X_s$ is the double cover of $\symE$ branched over $B$.  

Now consider the local system $\mathbb H:= R^2 f_\ast \Z(1)/(\text{torsion})$ of  free abelian groups of rank 12 on the complex manifold $\mathcal S(\C)$.  At a base point $s\in \mathcal S(\C)$ we have $\mathbb H_s \simeq H^2(\mathcal X_s,\Z(1))/(\text{torsion})$, and in fact the local system $\mathbb H$ is equivalent to a linear action of $\pi_1(\mathcal S(\C), s)$ on the free abelian group $\mathbb H_s$.  A priori, this is equivalent to a homomorphism
\begin{equation}\label{orig_mon_rep}
r\colon \pi_1(\mathcal S(\C),s)\to \Aut(\mathbb H_s) \simeq \GL_{12}(\Z),
\end{equation}
but we may refine this further. 

The group $H^0(\mathcal S,\mathbb H)$ of global sections of $\mathbb H$ contains a rank 2 subgroup $\mathbb I$; its fiber $\mathbb I_s$ corresponds to the subgroup of $\mathbb H_s$ generated by the canonical and Albanese classes on $\mathcal X_s$.  We may view $\mathbb I$ as a constant local subsystem of $\mathbb H$, which in terms of the monodromy action means that $\pi_1(\mathcal S(\C), s)$ fixes the subspace $\mathbb I_s$ inside $\mathbb H_s$.  Let $\psi\colon \mathbb H\otimes\mathbb H \to \Z(2)$ denote the morphism of local systems obtained by the cup product, which corresponds to the cup product form $\psi_s$ on $\mathbb H_s$ at a fiber.  Since the monodromy action of $\pi_1(\mathcal S(\C), s)$ respects $\psi_s$, we may define the local subsystem $\mathbb V:= \mathbb I^\perp$ to be the orthogonal complement of $\mathbb I$ in $\mathbb H$.\footnote{Note the slight change in notation from that used in the discussion preceding Theorem \ref{thmB} in \S\ref{intro}.  Here we are using $\mathbb V$ to denote a local system of free abelian groups, and its associated local system of $\Q$-spaces $\mathbb V\otimes_\Z \Q$ is what was denoted by $\mathbb V$ in \S1.}  Let us identify $\Aut(\mathbb V_s)$ with the subgroup of $\Aut(\mathbb H_s)$ of those transformations which fix $\mathbb I_s$ and preserve $\mathbb V_s$.  Then from this discussion, it follows that we may write the monodromy representation of $\mathbb H$ not in the form (\ref{orig_mon_rep}) but as
\begin{equation}
r\colon \pi_1(\mathcal S(\C),s)\to \Aut(\mathbb V_s) \cap \orth(\mathbb V_{s,\Q},\psi_s),
\end{equation}
where $\mathbb V_{s,\Q} = \mathbb V_s\otimes_\Z \Q$. The assertion of Theorem \ref{thmB} is that we cannot find a smaller $\Q$-algebraic subgroup of $\orth(\mathbb V_{s,\Q},\psi_s)$ containing the image of $r$, and in this sense the monodromy of $\mathbb H$ (or perhaps more appropriately of $\mathbb V$) is ``big''. 

Also note that the local system $\mathbb H_s$ underlies the variation of Hodge structure $R^2 f_\ast \Z(1)\otimes_\Z \mathcal O_{S}$, which has Hodge numbers $h^{-1,1}=h^{1,-1}=1$, $h^{0,0}=10$, and $h^{p,q}=0$ otherwise.  The subsystem $\mathbb I$ contains the (global) class which restricts to the canonical class on each fiber, which is ample everywhere by definition of $f\colon\mathcal X\to \mathcal S$.  It follows that the orthogonal complement $\mathbb V$ over $\mathcal S(\C)$ underlies a \emph{polarized} variation of Hodge structure (under the restriction of $\psi$) with Hodge numbers and $h^{-1,1}=h^{1,-1}=1$ and $h^{0,0}=8$.

To prove Theorem \ref{thmB}, we will focus on a particular subfamily of $f\colon\mathcal X\to\mathcal S$.  Recall the pencil $\Pi\subseteq \Dsys$ on $\symEone$, and the associated family of $Y\to \Pi$ of surfaces that was studied in \S\ref{PL_subsec}.  Let $g:Y_{sm}\to \Pi_{sm}$ denote the pullback of $Y\to \Pi$ to $\Pi_{sm}$.  By the identification of $\mathcal S_{a}$ with $\Dsys\setminus R$, we may view $\Pi_{sm}$ as a subvariety of $\mathcal S$.  Letting $\iota:\Pi_{sm}\hookrightarrow \mathcal S_0$ denote this inclusion, it follows that $Y_{sm}\to \Pi_{sm}$ is just the pullback of $\mathcal X\to \mathcal S$ via $\iota$.  Moreover, we an isomorphism of local systems $R^2 g_\ast \Z(1)/(\text{torsion}) \simeq \iota^\ast \mathbb H$ on $\Pi_{sm}(\C)$.  Thus if we choose the specific base point $b\in \Pi_{sm}$ such that $Y_b\simeq \mathcal X_{\iota(b)}\simeq X_1$ (where $X_1$ is the surface from Theorem \ref{thmA}), then the monodromy representation of $R^2 g_\ast \Z(1)/(\text{torsion})$ (which is essentially the representation $\rho$ from (\ref{orig_pencil_mon}), but with $\Z$-coefficients) factors through $\iota_\ast$, i.e., factors as
\[
\pi_1(\Pi_{sm}(\C),b) \stackrel{\iota_\ast}{\To} \pi_1(\mathcal S(\C),\iota(b)) \stackrel{r}{\To} \Aut(\mathbb V_{\iota(b)}) \cap \orth(\mathbb V_{{\iota(b)},\Q},\psi_{\iota(b)}).
\]
Now we note by Proposition \ref{orthog_decomp} that $\mathbb I_{\iota(b),\Q}\simeq I$ and $\mathbb V_{\iota(b),\Q}\simeq V^\ast$.  Hence in order to prove Theorem \ref{thmB}, it will suffice to show the following:

\begin{proposition}\label{pencil_mon_prop}
Consider the local system $R^2 g_\ast \Q(1)$ over $\Pi_{sm}(\C)$ obtained from the one-dimensional family $g:Y_{sm}\to \Pi_{sm}$.  The image of the associated monodromy representation
\begin{equation}\label{pencil_mon_map}
\pi_1(\Pi_{sm}(\C),b) \to \orth(V^\ast,\psi_b)
\end{equation}
is Zariski-dense in $\orth(V^\ast,\psi_b)$.
\end{proposition}

To prove this, we will need a group-theoretic result. The following statement is a simple combination of a result of Deligne and a technique from Lefschetz's study of the monodromy of hyperplane sections:

\begin{lemma}\label{group_lemma}
Let $G$ be a group with finite generating set $\set{g_i}$ and let $\pi:G\to \orth(V,\psi)$ be a finite-dimensional complex representation of $G$, polarized by a nondegenerate symmetric bilinear form $\psi$.  Suppose that:
\begin{enumerate}
\item[(a)] For each $i$ there is an element $v_i\in V$ such that $\psi(v_i,v_i)=2$, and $\pi(g_i)$ is a reflection
\begin{equation}\label{reflec}
\pi(g_i)\colon x\mapsto x-\psi(x,v_i)v_i
\end{equation}
through the hyperplane orthogonal to $v_i$.  Furthermore, the collection $\set{v_i}$ spans $V$.
\item[(b)] The representation $\pi$ factors through a quotient group $H$ such that the images of all $g_i$ are conjugate in $H$.
\end{enumerate}
Then the image of $\pi$ is either finite or Zariski-dense in $\orth(V,\psi)$.
\end{lemma}

\begin{proof}
Fixing some $i$, we will first show that the orbit of $v_i$ under $G$ contains $\pm v_j$ for any other $j$, by using precisely the same argument as in \cite[\S7.6]{Lamot}.  First, by (b) we may factor $\pi:G\to \orth(V,\psi)$ as $G \stackrel{q}{\twoheadrightarrow} H \stackrel{\pi'}{\to} \orth(V,\psi)$, and we have
\[
q(g_i) = h^{-1} q(g_j) h
\]
for some $h\in H$.  Thus $\pi'(h)\pi(g_i) = \pi(g_j)\pi'(h)$.  Using (\ref{reflec}), this translates into
\begin{equation}\label{insig}
\psi(x,v_i)\pi'(h)v_i = \psi(\pi'(h)x,v_j)v_j
\end{equation}
for any $x$ in $V$.  As $\psi$ is nondegenerate on $V$, we may choose $x$ so that $\psi(x,v_i)\neq 0$, in which case (\ref{insig}) gives $\pi'(h) v_i = c v_j$ for some scalar $c$.  Using (\ref{insig}) again, one has
\begin{eqnarray*}
\psi(\pi'(h)x,v_j)v_j &=& \psi(x,v_i)\pi'(h)v_i \\
&=& \psi(\pi'(h)x,\pi'(h)v_i))\pi'(h)v_i \\
&=& c^2 \psi(\pi'(h)x,v_j)v_j,
\end{eqnarray*}
which implies $c=\pm 1$.  Thus there exists $g\in G$ such that $\pi(g)v_i = \pi'(h)v_i = \pm v_j$.

After perhaps replacing some $v_j$ by $-v_j$, we may assume that all elements in $\set{v_i}$ belong to a single orbit $\mathcal O$ under $G$, with each $v\in \mathcal O$ satisfying $\psi(v,v)=2$.  Moreover, the image of $\pi$ is a subgroup of $\orth(V,\psi)$ containing all reflections $x\mapsto x-\psi(x,v)v$, with $v\in \mathcal O$.  Indeed, if $\pi(g)v_i=w$, then the reflection $x\mapsto x-\psi(x,v)v$ is equal to $\pi(gg_ig^{-1})$ by (\ref{reflec}).  Now we may apply \cite[Lemme $4.4.2^s$]{Del2} to conclude that the image of $\pi$ is either finite or Zariski-dense in $\orth(V,\psi)$.
\end{proof}

\noindent\textit{Proof of Proposition \ref{pencil_mon_prop}}.  Corresponding to the inclusions $\Pi_{sm}\hookrightarrow \mathcal S_a \hookrightarrow \mathcal S$, the monodromy representation (\ref{pencil_mon_map}) factors as
\[
\pi_1(\Pi_{sm}(\C),b) \to \pi_1(\mathcal S_a(\C),b) \to \orth(V^\ast,\psi_b).
\]
But $\mathcal S_a$ may be identified with the open subset $\Dsys\setminus R$ of smooth elements in the linear system $\Dsys$ on $\symEone$.  Recall by Proposition \ref{pencil_exists} that $R_0$ has codimension one in $\Dsys$, with exactly one irreducible codimension one component $R_0\subseteq R$.  By \cite[\S7.4, 7.5]{Lamot}, we know that (i) the homomorphism $\pi_1(\Pi_{sm},b) \to \pi_1(\Dsys\setminus R_0,b)$ is surjective and (ii) the images in $\pi_1(\Dsys\setminus R_0,b)$ of the 32 homotopy classes $\gamma_i\in \pi_1(\Pi_{sm},b)$ are all conjugate.  Since the complex codimension of $R\setminus R_0$ in $\Dsys$ is at least two, the map $\pi_1(\mathcal S_a,b)\simeq \pi_1(\Dsys\setminus R,b)\to \pi_1(\Dsys \setminus R_0,b)$ is an isomorphism (see \cite[Prop.\ 4.1.1]{Dimca}), and therefore the statements (i) and (ii) hold with $\pi_1(\mathcal S_a,b)$ in place of $\pi_1(\Dsys\setminus R_0,b)$.

Thus if we extend scalars to $\C$ and replace the cup product form $\psi_b$ with $-\psi_b$, then the preceding paragraph and the Picard-Lefschetz theory of \S\ref{PL_subsec} allow us to apply Lemma \ref{group_lemma} to conclude that the image of the monodromy representation (\ref{pencil_mon_map}) is either finite or Zariski-dense in $\orth(V^\ast_\C,\psi_b)$.  But we may rule out the first possibility, since the cup product form $\psi_b$ polarizes the representation on $V^\ast$ and has indefinite signature $(2,8)$.  Hence the image is dense in $\orth(V^\ast_\C,\psi_b)$, and therefore also dense in $\orth(V^\ast,\psi_b)$. \qed

\section{Applications to $\ell$-adic Cohomology}\label{tate_sec}

We are now in a position to prove Theorem \ref{thmC}, which we do using the axiomatic framework for polarized families of $p_g=1$ surfaces laid out in \S5 of \cite{Lyo}.  This framework is based closely upon a similar one due to Andr\'e \cite{Andre} (which in turn grows out of ideas of Deligne \cite{Del-K3} and \cite{Rapo}), but differs in an important way.  Andr\'{e} requires that the period map (arising from the primitive part of the middle singular cohomology groups) of the family is a submersion at some point.  Such a requirement can be satisfied for K3 surfaces, abelian surfaces, or some $p_g=1$ surfaces of general type \cite{Cat-2,Tod}, but not for the smooth family $f:\mathcal X\to\mathcal S$ of surfaces studied here: The base $\mathcal S$ (which is shown in \cite{Cat} to locally realize the Kuranishi family of the fiber at any point) has dimension 7, but the period domain for the Hodge structures $\mathbb V_s\subseteq H^2(\mathcal X_s,\Q(1))$ has dimension 8.  The framework in \cite{Lyo} relaxes this condition on the period map, and in particular applies to $f:\mathcal X\to\mathcal S$.

\subsection{}
Here is the general setup of this framework.  Let $Y$ be a smooth projective geometrically connected surface over the field $k_0$ such that $p_g=1$.  Then the weight zero integral Hodge structure
\[
H_\Z := H^2(Y_\C,\Z(1))/(\text{tors})
\]
will have Hodge numbers $h^{-1,1}=h^{1,-1}=1$, $h^{0,0}>1$, and $h^{p,q}=0$ otherwise.  We let $\theta$ denote the bilinear form on $H_\Z$ given by the cup product.  Suppose there is a sublattice $\Omega \subseteq H_\Z$ with the following properties:
\begin{itemize}
\item[(i)] There is a finite collection of effective divisors $D_1,\ldots, D_m$ on $Y$ that are defined over $k_0$ and whose cycle classes $[D_1],\ldots,[D_m]$ give a basis of $\Omega$.
\item[(ii)] The divisor $D_1$ is ample.
\end{itemize}
Let $V_\Z=\Omega^\perp \subseteq H_\Z$, which (with the restriction of the cup product form $\theta$) is a polarized integral Hodge structure by (ii).  This gives rise to the rational Hodge structure $V:=V_\Z\otimes_\Z \Q$.

Now let $\pi: \mathcal Y\to\mathcal T$ be a smooth projective family of surfaces defined over $k_0$ such that $\mathcal T$ is smooth and geometrically connected.  We will assume that this family satisfies four axioms.  The first two axioms essentially posit that the properties (i) and (ii) of $Y$ above can be extended to some family:
\begin{itemize}
\item[(A1)] For some point $t\in\mathcal T(k_0)$, we have a $k_0$-isomorphism $\mathcal Y_t\simeq Y$.
\item[(A2)] There exist effective divisors $\mathcal D_1,\ldots,\mathcal D_m$ on $\mathcal Y$ that are flat over $\mathcal T$ and whose pullbacks to $\mathcal Y_t\simeq Y$ are numerically equivalent to the divisors $D_1,\ldots,D_m$.  Moreover, the pullback of $\mathcal D_1$ to every fiber is ample.
\end{itemize}
Referring to (A1), we will abuse notation by also using $t$ to denote the associated point in $\mathcal T(\C)$ arising from our fixed embedding $k_0\hookrightarrow \C$.

Inside $\mathbb H_\Z:= R^2 (\pi_\C)^\ast \Z(1)/(\text{tors})$, which is an integral variation of Hodge structure on $\mathcal T_\C$, the cycle classes of the divisors $\mathcal D_1,\ldots, \mathcal D_m$ give rise to a constant subvariation.  If we take the orthogonal complement of this subvariation with respect to the cup product $\phi$ on $\mathbb H_\Z$, we obtain a polarized integral variation of Hodge structure that we may denote as $\mathbb V_\Z$.  Note by (A1) and (A2) that the isomorphism $\mathcal Y_t\simeq Y$ induces an isomorphism of Hodge structures $\mathbb H_{\Z,t}\simeq H_\Z$ and an isomorphism of polarized Hodge structures $(\mathbb V_{\Z,t},\phi_t)\simeq (V_\Z,\theta)$.  Let $\mathbb V:= \mathbb V_\Z \otimes_\Z \Q$.  The remaining two axioms replace the condition in \cite{Andre} on the period map of the family $\pi:\mathcal Y\to\mathcal T$ (and in particular are implied by such a condition):
\begin{itemize}
\item[(A3)] There exists $u\in\mathcal T(\C)$ such that the Hodge structure $\mathbb V_u$ contains nontrivial algebraic classes.

\item[(A4)] The image of the monodromy representation
\[
r: \pi_1(\mathcal T(\C),t)\to \orth(\mathbb V_t,\phi_t)\simeq \orth(V,\theta)
\]
contains a Zariski-dense index subgroup of $\text{SO}(V,\theta)$.
\end{itemize}
As noted in the introduction, these axioms may be used to show that the correspondence between the surface $Y$ and its Kuga-Satake variety (which need only be considered up to isogeny) is motivated.  This is done in \S5 of \cite{Lyo}, the culmination of which is the following:

\begin{theorem}\label{abstract-thmC}
Assuming the axioms (A1) through (A4), the following hold:
\begin{enumerate}
\item The $\ell$-adic representation of  $\Gal(k/k_0)$ acting upon $H^2(Y_k,\Q_\ell(1))$ is semisimple.
\item (Tate Conjecture) Any class of $H^2(Y_k,\Q_\ell(1))$ fixed by an open subgroup of $\Gal(k/k_0)$ is a $\Q_\ell$-linear combination of classes arising from divisors on $Y_k$.
\end{enumerate}
\end{theorem}

\begin{remark}\label{no-a3}
In view of the result of Green and Oguiso cited in Remark \ref{period-map}, (A3) will follow from (A4) as long as $h^{0,0}(V)=h^{0,0}(\mathbb V)>0$, since the variation $\mathbb V$ will be nontrivial.  On the other hand, when $h^{0,0}(V)=0$ then all fibers of $\mathbb V$ are CM Hodge structures, and so $\mathbb V$ is a trivial variation.  But in that case the Tate Conjecture is immediate for $Y$, since one would know that $\rho(Y)=h^{1,1}(Y)$. 
\end{remark}

\subsection{}

We will now prove Theorem \ref{thmC}.  First we note that both statements will be true over $k_0$ if they are true over finite extension of $k_0$.  Hence at various points in the argument below (in particular, when noting that the axioms (A1) through (A4) hold) we may replace $k_0$ by a finite extension when convenient.

Let $X$ be a surface with $p_g=q=1$, $K^2=2$, and $K$ ample.  Let $D_1:=K$ be the canonical divisor of $X$ and let $D_2$ be the class of an Albanese fiber.  By Corollary \ref{family}, $X$ is isomorphic over $k_0$ to a fiber $\mathcal X_s$ of the family $f:\mathcal X\to\mathcal S$ for some $s\in\mathcal S(k_0)$.  It follows from the same corollary that the divisors $D_1$, $D_2$ on $X$ arise as the restriction of divisors on $\mathcal X$ that are flat over $\mathcal S$, and also that the first of these global divisors restricts to an ample class on every fiber.  Axiom (A4) is immediate from Theorem \ref{thmB}, and axiom (A3) follows from Proposition \ref{nongen_picard} or Remark \ref{no-a3}.  This completes the proof.


\section*{Acknowledgements}

{The first author was partially supported by NSF grant DMS-0602191 as a participant in the 2011 REU at the University of Michigan.  A portion of the second author's work was supported by NSF grant DMS-0943832. We would like to thank Noam Elkies for help with calculations, Tom Graber and Dinakar Ramakrishnan for helpful discussions, and an anonymous referee for their insights and suggestions.}

\bibliography{BCH_surfaces-new}
\bibliographystyle{math}

\end{document}